\documentclass[a4paper,11pt]{article}
\usepackage[utf8]{inputenc}
\usepackage{amsmath}
\usepackage{mathtools}
\usepackage{mathrsfs} %\mathscr{}
\usepackage{amsfonts}
\usepackage{amssymb}
\usepackage{amsthm}
\usepackage{indentfirst}
\usepackage{graphicx}
\usepackage{booktabs}%
\usepackage[dvipsnames]{xcolor}
\usepackage[unicode, final, colorlinks=true]{hyperref} 
\usepackage{tikz, pgfplots}

\pgfplotsset{compat=1.18}
\usepgfplotslibrary{groupplots}
\usepgfplotslibrary{fillbetween}
\usetikzlibrary{backgrounds}
\usetikzlibrary{arrows.meta}
\pgfplotsset{plot coordinates/math parser=false}
\usetikzlibrary{external}
\usepgfplotslibrary{patchplots}

\usepackage[left=2cm, right=2.5cm]{geometry}

\def\cT{{\mathcal{T}}}

\def\Td{{\mathbb{T}^d}}

\def\softd{{\leavevmode\setbox1=\hbox{d}%
    \hbox to 1.05\wd1{d\kern-0.4ex{\char039}\hss}}}%cstocs
\newcommand{\ddt}{\partial_t}

\newcommand{\rT}{\mathbb{T}}
\newcommand{\brho}{\bar \rho}
\newcommand{\bc}{\bar c}

\newcommand{\ddx}{\partial_x}
\newcommand{\ddy}{\partial_y}

\newcommand {\startenv} {\vskip 0.05em
\begin{tabular}{||l}\parbox[t]{0.95\linewidth}}
\newcommand {\stopenv} {\end{tabular}\vskip 0.05em}

\newcommand{\F}{\mathcal{F}}

\providecommand{\keywords}[1]{\textit{Keywords:} #1}
\providecommand{\msc}[1]{\textit{2020 MSC:} #1}
\hypersetup{citecolor=teal, linkcolor=BrickRed, urlcolor=NavyBlue}
\newtheorem{theorem}{Theorem}[section]
\theoremstyle{definition}

\newtheorem{lemma}[theorem]{Lemma}

\theoremstyle{remark}
\newtheorem{remark}[theorem]{Remark}

\title{A posteriori error analysis of a positivity preserving scheme for the power-law diffusion Keller-Segel model}
\author{
  Jan Giesselmann$^{1}$ \and
  Niklas Kolbe$^2$\footnote{Corresponding author, email: {\tt kolbe@igpm.rwth-aachen.de}}
}

\date{
  \small
    $^1$Department of Mathematics,\\ Technical University Darmstadt, Dolivostraße 15,\\ 64293 Darmstadt, Germany\\[5pt]
  $^2$Institute of Geometry and Practical Mathematics,\\ RWTH Aachen University, Templergraben 55,\\ 52062 Aachen, Germany\\
   \smallskip
   \today
}

\begin{document}

\maketitle

\abstract{We study a finite volume scheme approximating a parabolic-elliptic Keller-Segel system with power law diffusion with exponent $\gamma \in [1,3]$ and periodic boundary conditions. We derive conditional a posteriori bounds for the error measured in the $L^\infty(0,T;H^1(\Omega))$ norm for the chemoattractant and by a quasi-norm-like quantity for the density. These results are based on stability estimates and suitable conforming reconstructions of the numerical solution.
We perform numerical experiments showing that our error bounds are linear in mesh width and elucidating the behaviour of the error estimator under changes of $\gamma$.
}


\msc{Primary 65M15,
Secondary 65M08, 35K40}

\section{Introduction}
The Keller-Segel system is the prototype of a large class of non-linear aggregation diffusion equations that is ubiquitous in continuum models of populations occurring, e.g., in mathematical biology, gravitational collapse and statistical mechanics, see~\cite{Carrillo2022}. The original parabolic-elliptic Keller-Segel system has been proposed as a model for chemotactic movement of bacteria under the assumption that the chemoattractant diffuses much faster than the bacteria. An extensive overview on the history and basic properties of this model can be found in~\cite{Horstmann,Arumgam2021}. One of its specific properties is that the solution may blow-up, i.e., all mass may concentrate in one point, in finite time~\cite{jager1992explossolutsystem}. This has motivated a plethora of analytical and numerical studies investigating the relation between initial data and blow-up. Introducing nonlinear diffusion in the system has been one of various proposed model refinements preventing blow-up, this modification moreover allows for a biological interpretation as volume filling effect during cell migration~\cite{calvez2006volumkellersegel}. Among many other applications, the Keller-Segel model has been extensively used as an essential component in the bio-medical modeling of tumor invasion of tissue, see e.g.~\cite{chaplain2005mathemmodelcancer, kolbe2021model}, in which different kinds of migration enter as well as the interplay of cells with the extracellular matrix and enzymatic activators.
One feature shared by the original model, its regularizations and models in applications is that 
%These regularization and application models share the feature with the original system that
solutions
develop 'spiky' highly localized densities, which makes the development and analysis of numerical schemes a challenging task.

There is a strong interest in developing and analyzing numerical methods for the Keller-Segel system. A main goal has been to derive schemes that retain structural properties of the PDE system such as non-negativity of the solution, conservation of mass, entropy dissipation or that preserve the asymptotic behavior in the parabolic to elliptic limit. Let us mention certain seminal and recent works. A temporal semi-discretisation that preserves positivity and is compatible with the parabolic to elliptic limit can be found in \cite{Liu2018}. A positivity and mass conservative upwind finite-element scheme was investigated in \cite{Saito2007} and error estimates were proven under the assumption that a strong solution exists. More recently, error estimates for a positivity preserving finite element scheme were derived in~\cite{Chen2022}.

Furthermore, discontinuous Galerkin (DG) methods have been used to solve Keller-Segel equations: A family of interior penalty semi-discrete discontinuous Galerkin methods for Keller-Segel equations was derived in~\cite{Epshteyn2009a} and error estimates were proven provided exact solutions are sufficiently regular. These results were extended to a fully discrete scheme in~\cite{Epshteyn2009b}. The local discontinuous Galerkin (LDG) method was applied to the 2D Keller-Segel model in~\cite{Li2017} and optimal error estimates were proven for smooth solutions. Subsequently, another LDG scheme was introduced that can be proven to be energy dissipative~\cite{Li2019}. A high order hybrid finite-volume finite-difference method was derived in~\cite{Chertock2018} and a dedicated scheme for the three dimensional case was derived in~\cite{Epshteyn2019}.

Another important class of numerical methods for Keller-Segel models are finite volume schemes for which positivity can be preserved easily. A finite volume scheme was derived and analyzed in \cite{Filbet2006} and can be shown to converge to a weak solution for sufficiently small initial data. Another finite volume scheme was proposed in \cite{Zhou2017} and error estimates were proven for sufficiently regular exact solutions.
 
This discussion shows that there has been a significant amount of work on deriving novel schemes and proving a priori error estimates that are valid as long as rather smooth exact solutions exist, see also the overview in~\cite{Chertock2019}. However, the time of existence of sufficiently smooth solutions is usually not known and it is very challenging to determine the actual error of a numerical simulation.

In addition to the mentioned high order (usually DG) schemes various approaches have been recently pursued in order to obtain efficient and accurate schemes: the mass-transport approach for the one-dimensional problem~\cite{blanchet2008convermasstrans, carrillo2019hybridmasstrans}, adaptive mesh refinement~\cite{Dudley2011, kolbe2022adaptrectanmesh} and moving mesh schemes \cite{Sulman2019,CKRW2019}. The latter are based on an \textit{error sensor} that determines regions in which the mesh should be refined --- in practice, a smoothed density gradient is used for this purpose. Apart from modeling applications a main motivation for these schemes is the investigation of structural properties of solutions, e.g., whether in two-species models blow-up of both densities is simultaneous, and underresolved simulations may give misleading results as noted in~\cite{CKRW2019}.

In this situation, a posteriori error estimates are a valuable complement to a priori error estimates that are currently missing in the literature. They offer several advantages:
They allow to move regularity requirements away from the exact solution to the numerical solution, thereby reducing the number of non-verifiable conditions in the analysis. They also provide computable error bounds so that one can check the accuracy of a given simulation and, finally, they provide a more rigorous way to define an error sensor for mesh adaptation. The goal of this paper is to obtain a posteriori error estimates that ensure that, as long as the numerical solution is 'well-behaved' --- in a sense that is to be made precise --- so is the exact solution and numerical and exact solution can be guaranteed to be close to each other.

We provide such an a posteriori error estimate for a scheme numerically approximating the parabolic-elliptic Keller-Segel system with non-linear power-law type diffusion that has been considered in \cite{calvez2017equil} and reads
\begin{equation}
 \label{KS}
 \begin{split}
 \partial_t \rho + \nabla \cdot (\rho \nabla c) - \Delta \rho^\gamma &=0\\
 \rho = c - \Delta c
 \end{split}
 \qquad \text{on } (0,T) \times \rT^d,
\end{equation}
where $\gamma \in [1,3]$ and $\rT^d$ denotes the $d$-dimensional flat torus, i.e., we consider periodic boundary data. The numerical method we consider is a new linearly implicit first-order finite volume scheme that builds on the upwind approach in~\cite{Filbet2006}. The scheme is introduced in this work and we show that it preserves positivity. While we allow for non-equidistant meshes in the one-dimensional case our study focuses on Cartesian meshes in two dimensions.
We believe that our analysis can, in principle, be extended to other (higher order) schemes and more complex models. Indeed, the companion paper~\cite{GiesselmannKwon} provides a posteriori error estimates for a DG scheme applied to the parabolic-elliptic Keller-Segel system with linear diffusion. We also believe that our analysis can be extended to more general meshes following the ideas in~\cite{Nicaise2005}.

Our analysis relies on an elliptic reconstruction of the numerical solution, cf.~\cite{Makridakis2003}, and a generalization of the Gronwall lemma that allows us to cope with strong nonlinearities. This generalized Gronwall lemma leads to \textit{conditional} error bounds, i.e., bounds that are only known to be valid when they are small enough; a condition that can be verified a posteriori.

The outline of the remainder of this paper is as follows: In Section 2 we state and prove two stability estimates for system~\eqref{KS}, one addressing the linear diffusion case ($\gamma=1$) and the other one the  nonlinear diffusion case ($1< \gamma \leq 3$). Section 3 introduces a finite volume scheme for \eqref{KS} and shows its positivity-preserving property. In Sections 4 and 5 we derive computable residual estimates for our scheme that are crucial for the evaluation of the a posteriori error estimates. Finally, we show numerical experiments for different $\gamma$ in Section 6, investigating the behavior of the a posteriori bounds and the stability conditions under global mesh refinement.

\section{Stability estimates}\label{sec:stability}
Let $(\rho, c)$ be a weak solution of~\eqref{KS} and $(\brho,\bc)$ a strong solution of the system
\begin{equation}\label{KSn-p}
 \begin{split}
 \partial_t \brho + \nabla \cdot (\brho \nabla \bc) - \Delta \brho^\gamma &=R_\rho\\
 \brho =  \bc -  \Delta \bc
 \end{split}
\qquad \text{ on } (0,T) \times \rT^d
\end{equation}
for some given function $R_\rho \in L^2(0,T; H^{-1}(\Td))$. The regularity requirements on $(\rho, c)$ and $(\brho,\bc)$ are made precise in the theorems below. Note that elliptic regularity implies $\|\bar  c\|_{H^2} \leq C_{ell} \|\bar \rho\|_{L^2}$ for some $C_{ell}>0$.
In the section at hand, we provide estimates for the difference $(\rho - \brho, c - \bc)$ in terms of the difference of their initial data and in terms of $R_\rho$. We provide two types of such stability estimates. The first one shown in Section~\ref{sec:classicalstability} is based on $L^\infty(0,T;L^2(\rT^d))$-norm estimates and works well for $\gamma=1$. The second one shown in Section~\ref{sec:nonlinearstability} is based on $L^\infty(0,T;H^{-1}(\rT^d))$-norm estimates and works well for larger values of $\gamma$. The situation we have in mind is that $\brho$ and $\bc$ are obtained as a reconstruction of a numerical solution and $R_\rho$ is the corresponding residual with respect to~\eqref{KS}.

\subsection{Stability estimate for the classical model}\label{sec:classicalstability}
In this section we provide a stability estimate for the classical model and therefore assume $\gamma =1$.
Subtracting \eqref{KSn-p} from \eqref{KS} and testing with $\rho- \brho$ we obtain
\begin{equation}
 \int_{\rT^d} (\rho - \brho) \partial_t (\rho - \brho) dx
 = \int_{\rT^d} (\rho - \brho) \Delta (\rho -\brho) -  (\rho - \brho) \nabla \cdot \left( \rho \nabla c - \brho \nabla \bc\right) - R_\rho(\rho - \brho) \, dx,
\end{equation}
which after integration by parts implies
\begin{align}
 \frac{d}{dt} \left[ \int_{\rT^d} \frac{1}{2} (\rho - \brho)^2 dx \right]  &+\int_{\rT^d}  |\nabla (\rho - \brho)|^2 \, dx \notag \\
  &=  \int_{\rT^d} \nabla (\rho - \brho) \rho \nabla (c - \bc) + \nabla (\rho - \brho) (\rho - \brho) \nabla \bc - R_\rho(\rho - \brho) \, dx \notag \\
  &=  \int_{\rT^d} \nabla (\rho - \brho) \brho \nabla (c - \bc) + \nabla (\rho - \brho) (\rho - \brho) \nabla \bc - R_\rho(\rho - \brho) \notag \\
 &\quad + \nabla (\rho - \brho) (\rho -\brho) \nabla (c - \bc) \, dx.  
 \end{align}
Using Cauchy-Schwartz's inequality we obtain
 \begin{align*}
 \frac{d}{dt} \left[\frac 12 \|  \rho - \brho\|_{L^2}^2  \right] 
   +   |\rho - \brho|_{H^1}^2  &\leq |\rho - \brho|_{H^1} \|\brho\|_{L^3}  |c - \bc|_{W^{1,6}} + |\rho - \brho|_{H^1} \|\rho - \brho\|_{L^2} \|\nabla \bc\|_{L^\infty}\\
   &\quad + \|R_\rho\|_{H^{-1}} \|\rho - \brho\|_{H^1}
 + |\rho - \brho|_{H^1} \|\rho -\brho\|_{L^2} \|\nabla (c - \bc)\|_{L^\infty}.
 \end{align*}
Provided the number of space dimensions satisfies $d \leq 3$ this implies
 \begin{align}\label{bla2}
   \frac{d}{dt} \left[\frac 12 \|  \rho - \brho\|_{L^2}^2  \right]  +  |\rho - \brho|_{H^1}^2 &\leq C_S |\rho - \brho|_{H^1} \|\brho\|_{L^3}  |c - \bc|_{H^2} + |\rho - \brho|_{H^1} \|\rho - \brho\|_{L^2} \|\nabla \bc\|_{L^\infty} \notag \\
   &\quad + \|R_\rho\|_{H^{-1}} \|\rho - \brho\|_{H^1}
 +C_S' |\rho - \brho|_{H^1} \|\rho -\brho\|_{L^2} \| c - \bc\|_{H^3} \notag\\
      &\leq C_S |\rho - \brho|_{H^1} \|\brho\|_{L^3}  \|\rho - \brho\|_{L_2} + |\rho - \brho|_{H^1} \|\rho - \brho\|_{L^2} \|\nabla \bc\|_{L^\infty} \notag \\
   &\quad + \|R_\rho\|_{H^{-1}} \|\rho - \brho\|_{H^1}
 +C_{ell} C_S' |\rho - \brho|_{H^1} \|\rho -\brho\|_{L^2} \| \rho - \brho\|_{H^1},
 \end{align}
 where $C_S$ and $C_S'$ are the Lipschitz constants of the embeddings
 $ H^1 \hookrightarrow L^6$ and $H^2 \hookrightarrow L^\infty$ respectively.
 Using Young's inequality, we obtain
 \begin{align}\label{3}
 \frac{d}{dt} \left[\frac 12 \|  \rho - \brho\|_{L^2}^2  \right]  +  |\rho - \brho|_{H^1}^2  
 &\leq \frac14 |\rho - \brho|_{H^1}^2 + C_S^2  \|\brho\|_{L^3}^2  \|\rho - \brho\|_{L_2}^2
   + \frac14 |\rho - \brho|_{H^1}^2 \notag \\
 &\quad +  \|\rho - \brho\|_{L^2}^2 \|\nabla \bc\|_{L^\infty}^2 + \|R_\rho\|_{H^{-1}}^2 + \frac14 \|\rho - \brho\|_{H^1}^2 \notag \\
   &\quad +  |\rho - \brho|_{H^1}  C_S' \|\rho -\brho\|_{L^2} \| \rho - \brho\|_{H^1}.
 \end{align} 
 
 In the next step, we set
 \begin{equation}
 \begin{split}
  y_1 (t) &= \frac 1 2 \|  \rho (t,\cdot) - \brho (t,\cdot)\|_{L^2}^2,   \\
  y_2(t) &= \frac 1 4 |  \rho (t,\cdot) - \brho (t,\cdot)|_{H^1}^2, \\
  y_3 (t)&=  |\rho(t,\cdot) - \brho(t,\cdot)|_{H^1}  C_S' \|\rho(t,\cdot) -\brho(t,\cdot)\|_{L^2} \| \rho(t,\cdot) - \brho(t,\cdot)\|_{H^1}, \\
  a_1(t) &=  2 C_S^2 \|\brho(t,\cdot)\|_{L^3}^2  +  2 \|\nabla \bc(t,\cdot)\|_{L^\infty}^2 + \frac 1 2.
 \end{split}
\end{equation}
Then, we integrate \eqref{3} in time from $0$ to $T'$ and get %use Young's inequality to obtain
\begin{equation}\label{near:gron}
 y_1(T') +  \int_0^{T'} y_2(t) dt 
 \leq
 y_1(0)  + \int_0^{T'} a_1(t) y_1(t)  +\| R_\rho\|_{H^{-1}}^2 + y_3(t)  dt.
\end{equation}

We have, using Young's inequality,
\begin{align*}  y_3(t) &\leq
C_S' |\rho - \brho|_{H^1}^2 \|\rho -\brho\|_{L^2} + C_S' |\rho - \brho|_{H^1} \|\rho -\brho\|_{L^2}^2\\
&\leq 4 \sqrt{2} C_S' y_2 \sqrt{y_1} + 4 C_S' \sqrt{y_2} y_1
\leq  (4 \sqrt{2} + 2)\, C_S'  \sqrt{y_1(t)} (y_1(t) + y_2(t) )
\end{align*}
and consequently
\begin{equation}\label{est:y3}
  \int_0^{T'} y_3(t) dt = (4 \sqrt{2} + 2) \, C_S' \sup_{t\in[0,T']} \sqrt{y_1(t)}  \int_0^{T'} y_1(t) + y_2(t) dt .
\end{equation}
 Equations \eqref{near:gron} and \eqref{est:y3} show that our analysis fits into the framework of \cite[Prop 6.2]{Bartels}
 with $y_1,y_2,y_3,a$ as above $B= 2C_S'$, $\beta = \frac 12$,
 \[
 A:= y_1(0)  + \int_0^{T} \| R_\rho\|_{H^{-1}}^2  dt                                                          
                                                      \quad \text{ and } \quad
  E_1 := \exp\left(  \int_0^{T}  a_1(t) dt\right).\]
Then, taking into account the regularity requirements in~\cite[Theorem 9]{Dragomir2003} we have the following conditional error estimate.

\begin{theorem}[Stability estimate for linear diffusion]\label{thm:stab1}
  Let $d\in\{1,2,3\}$, $\gamma=1$ and suppose that $\rho \in H^1(0, T; H^{-1}(\rT^d)) \cap L^2(0, T; H^{1}(\rT^d))$ and $c\in L^2(0, T; H^{1}(\rT^d))$ constitute a weak solution of the Keller--Segel model \eqref{KS}. Furthermore, for $\brho \in H^1(0, T; H^{-1}(\rT^d)) \cap L^2(0, T; H^{2}(\rT^d))$ let $R_\rho \in L^2(0,T; H^{-1}(\rT^d))$ and $\bar c \in L^2(0, T; H^{2}(\rT^d)) \cap L^2(0, T; W^{1,\infty}(\rT^d))$ satisfy \eqref{KSn-p} in terms of a strong solution. Then, provided
\begin{equation}\label{eq:classicstabcondition}
8 A E_1 ( 8 (4 \sqrt{2} + 2) C_S' (1 +T) E_1)^2 \leq 1
\end{equation}
%\nk{corrected from $4 C_S'$} 
is satisfied, it holds
\begin{align}\label{eq:classicmainestimate}
 \sup_{t \in [0,T]}  \frac 1 2 \|  \rho (t,\cdot) - \brho (t,\cdot)\|_{L^2}^2  &+ \frac 1 4 \int_0^T |  \rho (t,\cdot) - \brho (t,\cdot)|_{H^1}^2 dt
 \notag\\
 &\leq 8 \left(   \frac 1 2 \|  \rho (0,\cdot) - \brho (0,\cdot)\|_{L^2}^2 + \int_0^{T} \| R_\rho\|_{H^{-1}}^2  dt   \right) \notag \\
&\qquad  \exp \left( \int_0^T 2 C_S^2 \|\brho(t,\cdot)\|_{L^3}^2  +  2 \|\nabla \bc(t,\cdot)\|_{L^\infty}^2 + \frac 1 2 dt\right).
\end{align}
\end{theorem}

Note that, $\|\nabla \bc(t,\cdot)\|_{L^\infty}^2$ in \eqref{eq:classicmainestimate} is bounded up to a constant by  $\|\brho(t,\cdot)\|_{L^{3+\epsilon}}^2$ for any $\epsilon>0$.
 %%%%%%%%%%%%%%%%%%%%%%%%%%%%%%%%%%%%%%%%%%%%%%%%%%%%%%%%%%%%%%%%%%%%%%%%%%%%%%%%%%%%%%%%%%%%%%%%%%%%%%%%%
\subsection{Stability estimate for the power-law diffusion model}\label{sec:nonlinearstability}
In this section we show a stability estimate in case of nonlinear diffusion and assume that
$1 < \gamma \leq 3$. Subtracting \eqref{KSn-p} from \eqref{KS} and testing with $c- \bc$ we obtain
 \begin{equation}
  \int_{\rT^d} (c - \bc)  \partial_t (\rho - \brho)\, dx  = \int_{\rT^d} (c - \bc) \Delta (\rho^\gamma -\brho^\gamma) -  (c - \bc) \nabla \cdot \left( \rho \nabla c - \brho \nabla \bc\right) - R_\rho( c- \bar c)\, dx,
 \end{equation}
 which after integration by parts implies
 \begin{align}\label{eq:hm1evol}
  \frac{d}{dt} \left[ \int_{\rT^d} \frac{1}{2} (c - \bc)^2 + \frac{1}{2} |\nabla (c - \bc)|^2 dx \right]
   &=  \int_{\rT^d} \Delta (c - \bc)  (\rho^\gamma -\brho^\gamma) 
     + \nabla (c - \bc) \cdot \left( \rho \nabla c - \brho \nabla \bc\right)  \notag \\
   &\qquad -  R_\rho( c- \bar c)  \, dx \notag \\
  &= \int_{\rT^d}   (c- \bc) (\rho^\gamma -\brho^\gamma) -  (\rho- \brho)(\rho^\gamma - \brho^\gamma) + \rho |\nabla (c- \bc)|^2 \notag \\
  &\qquad + \nabla \bc \cdot \nabla (c - \bc) (\rho - \brho) - R_\rho( c- \bar c) \, dx.
 \end{align}

 We introduce the notations
 \begin{equation}
 \begin{split}
  z_1 (t) & \coloneq \frac 1 2 \|  c (t,\cdot) - \bc (t,\cdot)\|_{L^2}^2 + \frac 1 2 \|  \nabla (c (t,\cdot) - \bc (t,\cdot)  ) \|_{L^2}^2,  \\
  z_2(t) &:= \frac 1 2 \|  F(\rho,\brho)^{1/2} ( \rho (t,\cdot) - \brho (t,\cdot))\|_{L^2}^2,
  \end{split}
  \end{equation}
  where
   \begin{equation}\label{eq:F}
    F(\rho, \brho) := \left\{ \begin{array}{ccc}
                              ´\frac{\rho^\gamma - \brho^\gamma}{\rho - \brho} &&\text{if } \rho \not=\brho \\
                             \gamma \rho^{\gamma - 1}  && \text{otherwise}
                            \end{array}\right. .
   \end{equation}
The latter satisfies the bound
   \begin{equation}\label{eq:Flowerbound}
    F(\rho, \brho) =  \frac{\gamma}{\rho - \brho} \int_{\brho}^\rho s^{\gamma - 1} ds 
    \geq \left\{\begin{aligned}
                 \frac{\gamma}{2} (\rho^{\gamma -1 } + \brho^{\gamma -1}) & \text{ if} & 1 < \gamma \leq 2\\
                 \frac{\gamma}{2^{\gamma -1}} (\rho^{\gamma -1 } + \brho^{\gamma -1}) & \text{ if} & 2 \leq  \gamma \\
                \end{aligned}
  \right. ,
   \end{equation}
   which can be easily verified as follows: if the integrand is concave the integral is bounded from below by the trapezoidal rule and if the integrand is convex the integral is bounded from below by the midpoint rule.
 %  and an analogue formula for $\partial_{\brho} F(\rho, \brho)$ that
 %  \begin{equation}\label{eq:Flowerbound}
 %   F(\rho, \brho) \geq \max \{   \rho^{\gamma -1 }, \brho^{\gamma-1}\}.
 %  \end{equation}
Due to \eqref{eq:Flowerbound} we have
\begin{equation}\label{eq:z2lower}
 z_2 \geq \int \frac{c_\gamma}{2} (\rho^{\gamma -1 } + \brho^{\gamma -1})  (\rho - \brho)^2 \, dx\geq \int \frac{c_\gamma}{2} |\rho - \brho|^{\gamma + 1} \, dx 
\end{equation}
with $c_\gamma=\frac{\gamma}{2}$ if $\gamma \leq 2$ and $c_\gamma=\frac{\gamma}{2^{\gamma -1}}$ if $\gamma \geq 2$.
% We can also prove
% \begin{equation}
%  z_2 \geq \left\{ \begin{array}{ccc}
% C (\rho^{\gamma} - \rho^\gamma)^2 &:& |\rho - \brho | \leq 1\\
% C (\rho^{\gamma} - \rho^\gamma)^{\frac{\gamma +1 }{\gamma}} &:& |\rho - \brho | \geq 1
%  \end{array}
%  \right.
% \end{equation}
Then \eqref{eq:hm1evol} implies
\begin{align}\label{eq:defI}
  \frac{d}{dt} z_1 (t) + 2 z_2(t) &= \int_{\rT^d}   (c- \bc) (\rho^\gamma -\brho^\gamma) + \rho |\nabla (c- \bc)|^2 \notag \\
                                  &\quad + \nabla \bc \cdot \nabla (c - \bc) (\rho - \brho) +  R_\rho( c- \bar c) dx
 =: I_1 + I_2 + I_3 + I_4.
\end{align}
We will estimate the terms $I_1,\dots, I_4$ one by one. Before we do this, let us recall the following technical lemmas:

\begin{lemma}\label{lem:tech1}
 For $u, \bar u\in \mathbb{R}_{\geq  0}$ and $\alpha \geq 1$ the following inequalities holds: $|u^\alpha - \bar u^\alpha| \geq |u - \bar u|^\alpha$.
 \end{lemma}
\begin{proof}
 Assume w.l.o.g. $u> \bar u$ then we have
 \[ |u^\alpha - \bar u^\alpha| = \alpha \int_{\bar u}^u  s^{\alpha-1} ds
 \leq \alpha \int_{\bar u}^u  (s- \bar u)^{\alpha-1} ds  = |u - \bar u|^\alpha \]
 using the monotonicity of the mapping $s \mapsto s^{\alpha-1}$.
 \end{proof}

\begin{lemma}\label{lem:tech3}
For all $\rho, \brho >0$ and $\gamma >1$ the following inequality holds:
  \begin{equation}
  |\rho^{\gamma -1} - \brho^{\gamma-1}|
|\rho - \brho|  \leq \left( (\rho^\gamma - \brho^\gamma)(\rho - \brho)\right)^{\frac{\gamma}{\gamma+1}}.
\end{equation}
\end{lemma}
\begin{proof}
 Assume w.l.o.g. $\rho> \brho$.  Setting $u:= \rho^{\gamma -1} $ and $\bar u := \brho^{\gamma -1} $ we have for any $a \in [0,1]$, using Lemma~\ref{lem:tech1},
 \[
   (\rho^\gamma - \brho^\gamma)
   = ( u^{\frac{\gamma}{\gamma -1 }} - 
   \bar u^\frac{\gamma}{\gamma -1 })^{a} (\rho^\gamma - \brho^\gamma)^{1-a}
   \geq (u - \bar u)^{\frac{a\gamma}{\gamma-1}}
   (\rho - \brho)^{(1-a)\gamma}.
 \]
 It remains to show that we can choose $a$ such that
 \[\frac{a\gamma}{\gamma-1} = (1-a)\gamma +1 = \frac{\gamma +1 }{\gamma},  \]
 which can be easily verified taking $a = \frac{(\gamma+1)(\gamma-1)}{\gamma^2}$.
\end{proof}

We next estimate $I_1$. We observe that
\begin{equation}\label{eq:mean}
 |\rho^\gamma - \brho^\gamma| 
 = \gamma \tilde \rho^{\gamma-1} |\rho - \brho|
 \leq \gamma  \brho^{\gamma-1} |\rho - \brho|
+  \gamma |\rho^{\gamma-1} - \brho^{\gamma-1}| |\rho - \brho|
\end{equation}
for some $\tilde \rho$ between $\rho$ and $\brho$.
Thus using \eqref{eq:Flowerbound} we estimate
%We  define $\Omega= \{ x \in \mathbb{T}^d : |\rho(x) - \brho(x)| \leq 1\} $ and observe
\begin{align}\label{eq:I11}
I_1 &\leq \int_{\rT^d}\gamma |c- \bc| \brho^{\gamma-1}  |\rho - \brho|  
+\gamma |c- \bc| |\rho^{\gamma -1} - \brho^{\gamma-1}|
|\rho - \brho| \, dx \notag \\
&\leq
\gamma \sqrt{2 c_\gamma^{-1}z_2}\, \| \brho^{\frac{\gamma-1}{2}} \|_{L^3} \| c - \bar c \|_{L^6} 
+\gamma \int|c- \bc| |\rho^{\gamma -1} - \brho^{\gamma-1}|
|\rho - \brho| \, dx.
\end{align}
Applying Lemma \ref{lem:tech3}, Hölder's and Young's inequality in inequality \eqref{eq:I11} yields
\begin{align}
 I_1 &\leq  2 \gamma \sqrt{c_\gamma^{-1}} C_S \sqrt{z_2} \| \brho^{\frac{\gamma-1}{2}} \|_{L^3} \sqrt{z_1} + \gamma \| c - \bar c \|_{\gamma+1} (2 z_2)^\frac{\gamma}{\gamma+1} \notag \\
 % \leq C_S \sqrt{z_2} \| \brho^{\frac{\gamma-1}{2}} \|_{L^3} \| c - \bar c \|_{H^1} +
% \int | c - \bar c| | u - \bar u|^2  +  | c- \bar c| |u- \bar u|^\frac{\gamma}{\gamma-1}\\
&\leq  2 \gamma \sqrt{c_\gamma^{-1}} C_S \sqrt{z_2} \| \brho^{\frac{\gamma-1}{2}} \|_{L^3} \sqrt{z_1} + \gamma^{\gamma +1} C_Y \| c - \bar c\|_{L^{\gamma +1}}^{\gamma+1}   + \frac38 z_2 \notag \\
&\leq 2 \gamma \sqrt{c_\gamma^{-1}} C_S \sqrt{z_2} \| \brho^{\frac{\gamma-1}{2}} \|_{L^3} \sqrt{z_1} +  C_S C_Y  (\sqrt{2} \gamma)^{\gamma +1} z_1^{\frac{\gamma+1}{2}}  + \frac38 z_2,
\end{align}
where as above $C_S$ is the Lipschitz constant from the embedding $H^1 \hookrightarrow L^6$ and $C_Y = \frac{1}{\gamma+1} \left( \frac{16\gamma}{3\gamma+3} \right)^\gamma$ is a constant from Young's inequality.

The next step is to control $I_2$. %(This is the place where we cannot handle $\gamma >3$)
We have
\begin{align*}
 I_2 &\leq 2 \|\bar \rho\|_{L^\infty} z_1 + \int_{\rT^d} (\rho - \bar \rho) |\nabla c - \nabla \bc|^2\, dx \\
 &\leq 2\|\bar \rho\|_{L^\infty} z_1 + \| \rho - \bar \rho\|_{L^{\gamma +1}}  \| \nabla c - \nabla \bc\|_{L^2}
\| \nabla c - \nabla \bc\|_{L^{\frac{2\gamma +2 }{\gamma -1}}}.
\end{align*}
Assuming that $7/5 \leq \gamma \leq 3$ in the case $d=3$ we use the Sobolev embedding $W^{2,\gamma+1} \hookrightarrow W^{1,\frac{2\gamma+2}{\gamma-1}}$ together with elliptic regularity to bound 
$\| \nabla c - \nabla \bc\|_{L^{\frac{2\gamma +2 }{\gamma -1}}}$  by a multiple of $\| \rho - \bar \rho\|_{L^{\gamma +1}} $
and obtain using \eqref{eq:z2lower} and Young's inequality
\begin{align}
 I_2 &\leq 2 \|\bar \rho\|_{L^\infty}  z_1  + \tilde C_S \| \rho - \bar \rho\|_{L^{\gamma +1}}^2  \| \nabla c - \nabla \bc\|_{L^2} \notag \\
 &\leq 2 \|\bar \rho\|_{L^\infty}  z_1   + \tilde C_S \tilde c_\gamma \sqrt{2 z_1}  z_2^{\frac2{\gamma+1}} \leq 2 \|\bar \rho\|_{L^\infty}  z_1 + \tilde C_S C_Y^\prime z_1^{\frac{\gamma+1}{2(\gamma-1)}}  + \frac38 z_2,
\end{align}
where $\tilde C_S$ accounts for the Sobolev embedding, $\tilde c_\gamma= (\frac{c_\gamma}{2})^{\frac{\gamma+1}{2}}$ and $C_Y^\prime=\sqrt{2} \tilde c_\gamma\frac{\gamma-1}{\gamma+1}\left( \frac{16 \sqrt{2} \tilde C_S \tilde c_\gamma}{3\gamma+3}\right)^\frac{2}{\gamma-1}$ comes from Young's inequality.
The term $z_1^{\frac{\gamma+1}{2(\gamma-1)}} $ is higher order compared to $z_1$ so that it can be treated via the generalized Gronwall lemma \cite[Prop 6.2]{Bartels}. It is simply $z_1$ for $\gamma =3$. %(Probably) this term cannot be handled for $\gamma >3$. 

Our next step is to control $I_3$. We estimate
\begin{align}
 I_3 &= \int_{\rT^d} \nabla \bc \cdot \nabla ( c - \bc ) (c- \bc)  - \nabla \bc \cdot \nabla ( c - \bc) (\Delta c - \Delta \bc) \, dx \notag \\
 &\leq \| \nabla \bc \|_{L^3} \|\nabla c - \nabla \bc \|_{L^2}\| c - \bc\|_{L^6}
 + \int_{\rT^d} \nabla \bc \cdot \nabla \cdot \left( -\frac12 |\nabla (c- \bc)|^2   + \nabla (c- \bc) \otimes \nabla (c- \bc) \right) \, dx \notag \\
 &\leq 2 C_S \| \nabla \bc \|_{L^3}  z_1 + C \int_{\rT^d} |\Delta \bc|  | \nabla (c- \bc)|^2 \, dx \notag \\
& \leq 2 C_S \| \nabla \bc \|_{L^3}  z_1  + 2 C z_1 \| \Delta \bc\|_{L^\infty} 
 \leq 2 C_S \| \nabla \bc \|_{L^3}  z_1  + 2 C z_1 \| \brho \|_{L^\infty},%\\
 %\leq C_S \| \nabla \bc \|_{L^3}  z_1  + C C_S z_1 \| \brho \|_{H^2}
\end{align}
where we have used integration by parts in the second and elliptic regularity in the last inequality and introduced the constant $C=d/2$.
%We assume in the following that an upper bound for $\| \brho \|_{L^\infty}$ %or $\| \brho \|_{H^2}$ can be computed.This is fine if we use $\bar \rho$ defined by the (interpolated in time) finite element interpretation of $\rho_h$. 

%: For $\rho_h$ bounded from below by a positive constant this is fine since $\brho$ solves an elliptic problem with known coefficients and known right hand side.
%For $\rho_h$ arbitrarily close to zero (or even exactly zero) we cannot argue this way.

Finally, $I_4$ can be controlled via
\begin{equation}
 I_4 \leq \| R_\rho\|_{H^{-1}} \sqrt{2 z_1}.
\end{equation}

We combine our estimates for $I_1,\dots,I_4$ in \eqref{eq:defI} and obtain, using Young's inequality
\begin{align}\label{eq:estgk}
 \frac{d}{dt} z_1 + 2 \, z_2 &\leq
4 C_S^2 c_\gamma^{-1} \gamma^2 \, \| \brho^{\frac{\gamma-1}{2}} \|_{L^3}^2 z_1 +  C_S C_Y (\sqrt{2} \gamma)^{\gamma+1}\,z_1^{\frac{\gamma +1}{2}}  + z_2
+ 2 \|\bar \rho\|_{L^\infty}  z_1 + \tilde C_S C_Y^\prime  z_1^{\frac{\gamma+1}{2(\gamma-1)}} \notag \\
&\quad + 2 C_S \| \nabla \bc \|_{L^3}  z_1  + 2 C \| \brho \|_{L^\infty} z_1  +  \| R_\rho\|_{H^{-1}}^2 + \frac 12 z_1.
\end{align}
Since for any $a >0$ and $1 < \alpha < \beta$ it holds $a^\alpha < a + a^\beta$ we have
\begin{equation}
  z_1^{\frac{\gamma +1}{2}}  \leq z_1 + z_1^{\frac{\gamma+1}{2(\gamma-1)}} \quad \text{if }1<\gamma < 2, \qquad z_1^{\frac{\gamma+1}{2(\gamma-1)}}  \leq z_1 + z_1^{\frac{\gamma+1}2}  \quad \text{if }2\leq\gamma \leq 3.
\end{equation}

% Equation \eqref{eq:estgk} fits into the generalized Gronwall lemma \cite[Prop 6.2]{Bartels} since  
% \[ z_1^{\frac{\gamma +1}{2}}  \leq z_1 + z_1^{\frac{\gamma+1}{2(\gamma-1)}}\]
% which follows from  $1 < \frac{\gamma +1}{2}  < \frac{\gamma+1}{2(\gamma-1)}$ 
% %or $z_1^{\frac{\gamma+1}{2(\gamma-1)}}$ by $z_1 + z_1^{\gamma +1} $. Which one it is depends on whether $2(\gamma-1)$ is larger or smaller than $1$.

% We can proceed similarly for $2 \leq \gamma \leq 3$ and obtain using Young's inequality
% \begin{multline}\label{eq:estgg}
%  \frac{d}{dt} z_1 + z_2 \leq
%   C_S^2 \| \brho^{\frac{\gamma-1}{2}} \|_{L^3}^2 z_1 + \frac38 z_2 + C_SC_Y (z_1 + z_1^{\frac{\gamma+1}2})
%   +  \|\bar \rho\|_\infty z_1 \\
%   +\tilde C_S  C_Y z_1^{\frac{\gamma+1}{2(\gamma-1)}}
%   + C_S \| \nabla \bc \|_{L^3}  z_1  + C  z_1 \| \brho \|_{L^\infty}   +  \| R_\rho\|_{H^{-1}}^2 + z_1.
% \end{multline}
% %\nk{Konstanten aus Young's Ungleichung sollten prinzipiell noch korrigiert werden (auch in\eqref{eq:estgk}). Insbesondere können wir für den ersten Term rechts die skalierte Version nutzen, so dass $1/2$ in \eqref{near:gronk} in beiden Fällen passt.}
% Equation \eqref{eq:estgg} fits into the generalized Gronwall lemma since
% $z_1^{\frac{\gamma+1}{2(\gamma-1)}}  \leq z_1 + z_1^{\frac{\gamma+1}2}$.

%%%%%%%%%%%%%%%%%%%%%%%%%%%%%%%%%%%%%%%%%%%%%%%%%%%%%%%%%%%%%%%%%%%%%%%%%%%%%%%%%%%%%%%
Equation \eqref{eq:estgk} fits into the framework of \cite[Prop 6.2]{Bartels}.
%the specifics vary slightly between two cases:
% $1 < \gamma \leq 2$ and $2 \leq \gamma \leq 3$:
We define
\begin{equation}\label{eq:a}
  a_\gamma(t)\coloneqq
   4 C_S^2 c_\gamma^{-1} \gamma^2 \| \brho^{\frac{\gamma-1}{2}} \|_{L^3}^2 +  C_a(\gamma)
+ 2(C+1) \|\bar \rho\|_{L^\infty}
+ 2 C_S \| \nabla \bc \|_{L^3}   + \frac 12
\end{equation}
given the constant
\begin{equation}
C_a(\gamma)\coloneqq \left\{ \begin{array}{ccc}
                               C_S C_Y (\sqrt{2} \gamma)^{\gamma+1}
                               &\text{if}& 1< \gamma < 2 \\[5pt]
                               \tilde C_S C_Y^\prime &\text{if}& 2 \leq \gamma \leq 3
  % C_S^2 \| \brho^{\frac{\gamma-1}{2}} \|_{L^3}^2 +  \tilde C_S C_Y + (C+1) \|\bar \rho\|_\infty + C_S \| \nabla \bc \|_{L^3}   +1
%   \left\{ \begin{array}{ccc}
%             C_S^2 \gamma^2 \| \brho^{\frac{\gamma-1}{2}} \|_{L^3}^2 +  C_S C_Y \gamma^{\gamma+1}
% + (C+1) \|\bar \rho\|_\infty
% + C_S \| \nabla \bc \|_{L^3}   + \frac 14
% &\text{if}& 1< \gamma < 2 \\[5pt]
%                 C_S^2 \gamma^2\| \brho^{\frac{\gamma-1}{2}}  \|_{L^3}^2 + \tilde C_S C_Y^\prime + \tilde C_S C_Y  + C_S \| \nabla \bc  \|_{L^3}
%                  + C  \| \brho \|_{L^\infty} +1 &\text{if}& 2 \leq \gamma \leq 3,
     %             C_S^2 \| \brho^{\frac{\gamma-1}{2}} \|_{L^3}^2
%+ \|\bar \rho\|_\infty + C_Y
%+ C_S \| \nabla \bc \|_{L^3}   + C  \| \brho \|_{L^\infty}  + 1
%&:& \frac32 \leq \gamma \leq 2,\\

               \end{array}\right. .
\end{equation}

Then, we can integrate \eqref{eq:estgk} in time from $0$ to $T'$ and  %\nk{Wird hier tatsächlich Youngs Ungleichung benutzt? Ich denke, es fehlt vor dem letzten Term ein $B=\tilde C_S C_Y + C_S C_Y$} to
obtain
\begin{equation}\label{near:gronk}
 z_1(T') + \int_0^{T'} z_2(t) dt
 \leq
 z_1(0)  + \int_0^{T'} a_\gamma(t) z_1(t)  +\| R_\rho\|_{H^{-1}}^2 + B z_1(t)^{1+ \beta}  dt
\end{equation}
with
\begin{equation}
  B \coloneqq C_S C_Y (\sqrt{2} \gamma)^{\gamma+1} + \tilde C_S C_Y^\prime, \qquad
  \beta \coloneqq %\frac{\gamma+1}{2(\gamma-1)} -1
  \left\{ \begin{array}{ccc}
              \frac{3-\gamma}{2(\gamma-1)}
&\text{if}& 1 < \gamma < 2\\
             \frac{\gamma - 1}{2}
&\text{if}& 2 \leq \gamma \leq 3
               \end{array}\right. .
\end{equation}

 Equation \eqref{near:gronk} shows that our analysis fits into the framework of \cite[Prop 6.2]{Bartels}
 with $y_1:= z_1$, $y_2:= z_2$, 
 \[
 y_3 = Bz_1(t)^{1+\beta}, \qquad A:= z_1(0)  + \int_0^{T} \| R_\rho\|_{H^{-1}}^2  dt,
 \qquad E_\gamma := \exp\left(  \int_0^{T}  a_\gamma(t) dt\right),
\]
$a_\gamma$, $\beta$ and $B$ as above. Then, we have the following conditional error estimate.

\begin{theorem}[Stability estimate for nonlinear diffusion]
  Let either $d=1,2$ and $1<\gamma\leq3$ or $d=3$ and $7/5\leq \gamma \leq 3$. Suppose that $\rho \in H^1(0, T; H^{-1}(\rT^d)) \cap L^2(0, T; H^{1}(\rT^d))$ and $c\in C(0, T; H^{1}(\rT^d))$ constitute a weak solution of the power-law diffusion Keller--Segel model \eqref{KS}. Furthermore, for $\brho \in H^1(0, T; H^{-1}(\rT^d)) \cap L^2(0, T; H^{2}(\rT^d)) \cap L^1(0,T; L^\infty(\rT^d))$ let $R_\rho \in L^2(0,T; H^{-1}(\rT^d))$ and $\bar c \in L^2(0, T; H^{2}(\rT^d)) \cap C(0, T; H^{1}(\rT^d))$ satisfy \eqref{KSn-p} in terms of a strong solution. Then, provided
\begin{equation}\label{eq:stabilitycondition}
%%% Constants need to be checked  %%%
  8 A E_\gamma (8 B  (1 +T) E_\gamma)^{\frac1\beta} \leq 1
\end{equation}
is satisfied, it holds
\begin{multline}\label{mainestimate}
 \sup_{t \in [0,T]}  \frac 1 2 \|  c (t,\cdot) - \bc (t,\cdot)\|_{H^1}^2
 + \frac 1 2 \int_0^T   \int_{\mathbb{T}^d}    (\rho^\gamma - \brho^\gamma) (\rho - \brho)
 dxdt
 \\
 \leq 8 \left(   \frac 1 2 \|  c (0,\cdot) - \bc(0,\cdot)\|_{H^1}^2 + \int_0^{T} \| R_\rho\|_{H^{-1}}^2  dt   \right)
 \exp \left( \int_0^T a_\gamma(t) dt\right).
\end{multline}
\end{theorem}

\begin{remark}
 To make the estimate for the density difference resulting from \eqref{mainestimate} more transparent, we note that it can be expressed in the style of what is done for the gradient in $p$-Laplace problems \cite{Liu2001} and called \textit{quasi-norm}, i.e.
 \[
    |\rho^{\frac{\gamma+1}{2}} - \brho^{\frac{\gamma+1}{2}}|^2 \leq    (\rho^\gamma - \brho^\gamma) \, (\rho - \brho)
    \leq \frac{\gamma+1}{2}  |\rho^{\frac{\gamma+1}{2}} - \brho^{\frac{\gamma+1}{2}}|^2.
 \]
\end{remark}

\section{A positivity-preserving finite volume scheme}\label{sec:scheme}
We consider a discretization of the time domain with step sizes $\Delta t^n>0$, time instances $t^n=\sum_{i=1}^n \Delta t^i$ and a discretization of $\mathbb{T}^d$ for $d\in\{1,2\}$ on a grid with the mesh cells $K_1$, \dots, $K_N$. In the case $d=2$ the cells form an equidistant Cartesian mesh, whereas in the case $d=1$ we allow for nonuniform interval cells, details are given below. By $\rho_h^n$ we denote a piecewise constant approximation of \eqref{KS} at time $t^n$ consisting of the cell averages $\rho_i^n$.

\paragraph{The 1D scheme} In the case $d=1$ we introduce the cell interfaces $x_{i+1/2}$ for $i=0,1,\dots,N$, such that
$h_i= x_{i+1/2} - x_{i-1/2}>0$, $x_{-1/2}=0$ and $x_{N+1/2}=1$ and define the mesh cells $K_i=[x_{i-1/2}, x_{i+1/2}]$. 
We also define cell-midpoints $x_i:= \tfrac12 (x_{i+1/2} + x_{i-1/2})$ and define the function spaces
\begin{align}
 R_h &:= \{ r \in L^2( \rT^1, \mathbb{R})\,: \,  r|_{(x_{i-1/2} , x_{i+1/2} )} \text{ constant} \quad \forall i \},\\
 V_h &:= \{ v \in H^1(\rT^1, \mathbb{R})\,: \,  v|_{(x_{i} , x_{i+1} )} \text{ linear} \quad \forall i \},\\
 \tilde V_h &:= \{ v \in H^1(\rT^1, \mathbb{R})\,: \,  v|_{(x_{i-1/2} , x_{i+1/2} )} \text{ linear} \quad \forall i \}.
\end{align}

Starting from $\rho_h^0 \in R_h $ we compute successively $c_h^n \in V_h$ and $\rho_h^{n+1}\in R_h $
by solving
\begin{equation}\label{eq:cfem}
 \int_{\rT^1} \nabla c^n_h \nabla v_h + c_h^n v_h \, dx  = \int_{\rT^1} \tilde \rho^n  v_h \, dx \quad \forall v_h \in V_h,
\end{equation}
where $\tilde \rho^n \in V_h$ is determined by $\tilde \rho^n(x_i)=\rho^n_i$, the latter being given by the scheme
\begin{equation}\label{eq:scheme}
  \rho_i^{n+1} 
  = \rho_i^{n}
  -  \frac {\Delta t^n}{h_i} \left( \mathcal{F}_{i+1/2}^{n} - \mathcal{F}_{i-1/2}^{n}\right)
  + \frac{\Delta t^n}{h_i} \left(\mathcal{D}_{i+1/2}^{n,n+1} -  \mathcal{D}_{i-1/2}^{n,n+1}\right), %
\end{equation}
where $\rho_i^n$ denotes the value of $\rho_h^n$ on $(x_{i-1/2} , x_{i+1/2} )$.
% \begin{equation}
%   \frac {\rho_i^{n+1}-\rho_i^{n}}{\Delta t^n} + \frac {1}{h} \left[ \mathcal{F}_{i+1/2}^{n} - \mathcal{F}_{i-1/2}^{n}\right]
%     % - ((\bar \rho_{i+1/2}^n )^{\gamma-1} + (\bar \rho_{i-1/2}^n )^{\gamma-1}) \rho_{i}^{n+1}
%    = \Delta_h^\gamma[\rho_h^n, \rho_h^{n+1}] \rvert_{K_i},
%  \end{equation}
The numerical fluxes accounting for nonlinear diffusion are defined by
\begin{equation}\label{eq:dflux}
\mathcal{D}_{i+1/2}^{n, n+1} =  \gamma \, (\hat \rho_{i+1/2}^{n} )^{\gamma-1} \, \frac{\rho_{i+1}^{n+1} - \rho_{i}^{n+1}}{d_{i+1/2}}, 
\end{equation}
where we employ the averages
\begin{equation}
  \hat \rho_{i+1/2}^{n} = \frac {\rho_{i+1}^{n} + \rho_i^{n} }{2}.
\end{equation}
%Let $x_i$ denote the center of the cell $K_i$, then
and $d_{i+1/2}:=\frac{h_i + h_{i+1}}{2}$ refers to the distance between $x_i$ and $x_{i+1}$. For brevity of notation in the following computations we further introduce the notation
\begin{equation}\label{eq:dgamma}
   \Delta_h^\gamma[\rho_h^n, \rho_h^{n+1}] \rvert_{K_i} = \frac{1}{h_i} \left(\mathcal{D}_{i+1/2}^{n,n+1} -  \mathcal{D}_{i-1/2}^{n,n+1}\right).
\end{equation}
The advective numerical fluxes are given by
\begin{equation}\label{eq:numflux}
   \mathcal{F}_{i+1/2}^{n} = \partial_x  c_h^n(x_{i+1/2})^+ \rho_i^{n} -  \partial_x  c_h^n(x_{i+1/2})^- \rho_{i+1}^{n}
\end{equation}
with the positive and negative part defined as $x^+=\max\{0, x\}$ and $x^-= - \min\{x, 0\}$.

\paragraph{The 2D scheme} In the case $d=2$ the mesh cells for $\rho_h$ are given by the squares $K_{i}=K_{j,k}=[x_{j-1/2}, x_{j+1/2}] \times [y_{k-1/2}, y_{k+1/2}]$, where the single and double index notation relates as $i=(k-1)n+(j-1)$ for $i,j=1,\dots,n$ and $N=n^2$. In analogy to the 1D scheme we consider $x$- and $y$-coordinates of the cell interfaces in horizontal and vertical direction, which due to the assumption of equidistant cells of side length $h$, have the simple structure $x_{i+1/2} = y_{i+1/2} = (i-1) h$ for $i=0,1,\dots,N$. The time- and space dependent cell averages $\rho_{j,k}^n$ are governed by the scheme
\begin{multline}\label{eq:scheme2d}
  \frac {\rho_{j,k}^{n+1}-\rho_{j,k}^{n}}{\Delta t^n} + \frac{1}{h} \left(  \mathcal{F}_{j+1/2, k}^{n} -  \mathcal{F}_{j-1/2, k}^{n}\right) + \frac{1}{h} \left(  \mathcal{F}_{j, k+1/2}^{n} -  \mathcal{F}_{j, k-1/2}^{n}\right)
\\ = \frac{1}{h} \left( \mathcal{D}_{j+1/2, k}^{n, n+1} -  \mathcal{D}_{j-1/2, k}^{n, n+1} \right) + \frac{1}{h} \left( \mathcal{D}_{j, k+1/2}^{n, n+1} -  \mathcal{D}_{j, k-1/2}^{n, n+1} \right).
\end{multline}
In order to define $c_h$ we introduce a triangulation $\cT_h$ of $\rT^2$ by considering the dual quadrilateral mesh consisting of cells $[x_{j}, x_{j+1}] \times [y_{k}, y_{k+1}]$ and dividing each square into a lower-left and an upper-right triangle.
We define
\[ 
 V_h =\{ v \in H^1(\rT^2, \mathbb{R}) \, : \, v|_T \text{ is linear } \forall T \in \cT_h\}
\]
and $c_h^n \in V_h$ by 
\[
 \int_{\rT^2} \nabla c_h^n \nabla v_h + c_h^n v_h \, dx \, dy = \int_{\rT^2} \tilde \rho^n v_h \, dx \, dy,
\]
where $\tilde \rho^n$ is the unique element of $V_h$ with $\tilde \rho^n(x_j,y_k)=\rho^n_{jk}$.
% \begin{equation}\label{eq:scheme2d}
%   \frac {\rho_{i}^{n+1}-\rho_{i}^{n}}{\Delta t^n} 
%     % - ((\bar \rho_{i+1/2}^n )^{\gamma-1} + (\bar \rho_{i-1/2}^n )^{\gamma-1}) \rho_{i}^{n+1}
%    = \sum_{\alpha\in\{x, y \}}\Delta_h^{\alpha,\gamma}[\rho_h^n, \rho_h^{n+1}] \rvert_{K_i} - \frac {1}{h} \left[ \mathcal{F}_{i+1/2, j}^{\alpha,n} - \mathcal{F}_{i-1/2, j}^{\alpha,n}\right]
% \end{equation}
The numerical fluxes analogue to  \eqref{eq:dgamma} and \eqref{eq:numflux} are direction dependent. To discretize the diffusion terms the numerical fluxes
\begin{align}
\mathcal{D}_{j+1/2, k}^{n, n+1} = \gamma  \, (\hat \rho_{j+1/2, k}^{n} )^{\gamma-1} \, \frac{\rho_{j+1, k}^{n+1} - \rho_{j,k}^{n+1}}{h}, \qquad  \mathcal{D}_{j, k+1/2}^{n, n+1}= \gamma \, (\hat \rho_{j,k+1/2}^{n} )^{\gamma-1} \, \frac{\rho_{j,k+1}^{n+1} - \rho_{j,k}^{n+1}}{h}
\end{align}
are used together with the  averages
\[
  \hat \rho_{j+1/2,k}^{n} = \frac {\rho_{j+1,k}^{n} + \rho_{j,k}^{n} }{2}, \qquad  \hat \rho_{j,k+1/2}^{n} = \frac {\rho_{j,k+1}^{n} + \rho_{j,k}^{n} }{2}.
\]
Like \eqref{eq:dgamma} in the 1D scheme we introduce the abbreviations
\begin{align*}
  \Delta_{x,h}^{\gamma}[\rho_h^n, \rho_h^{n+1}] \rvert_{K_i} &= \frac{1}{h} \left( \mathcal{D}_{j+1/2, k}^{n, n+1} -  \mathcal{D}_{j-1/2, k}^{n, n+1} \right), \\
  \Delta_{y,h}^{\gamma}[\rho_h^n, \rho_h^{n+1}] \rvert_{K_i} &= \frac{1}{h} \left( \mathcal{D}_{j, k+1/2}^{n, n+1} -  \mathcal{D}_{j, k-1/2}^{n, n+1} \right).
\end{align*}
To discretize the advection terms, we denote the centers of the intervals $[x_{j-1/2}, x_{j+1/2}]$ and $[y_{k-1/2}, y_{k+1/2}]$ by $x_j$ and $y_k$, respectively and use the numerical fluxes
\begin{align*}
  \mathcal{F}_{j+1/2, k}^{n} &=  \partial_x c_h^n (x_{j+1/2}, y_k)^+ \rho_{j,k}^{n} -  \partial_x c_h^n(x_{j+1/2}, y_k)^- \rho_{j+1,k}^{n}, \\
   \mathcal{F}_{j, k+1/2}^{n} &=  \partial_y c_h^{n} (x_{j}, y_{k+1/2})^+ \rho_{j,k}^{n} - \partial_y c_h^{n} (x_{j}, y_{k+1/2})^- \rho_{j,k+1}^{n},
\end{align*}
where the points $(x_{j+1/2}, y_k)$ and $(x_{j}, y_{k+1/2})$ are located at the center of an interface between two adjacent mesh cells.
Note that  $\partial_x c_h^n (x_{j+1/2}, y_k)$ and $\partial_y c_h^{n}(x_{j}, y_{k+1/2})$ are well defined.

\begin{remark} It is not clear, whether our definition of the reconstruction $\tilde \rho$ achieves optimal a posteriori estimates for the proposed finite volume scheme. An alternative would be a biliner reconstruction on a quadrilateral mesh.
\end{remark}

The introduced finite volume scheme employs an implicit discretization of the nonlinear diffusion terms. However, a time step of the scheme does not require the solution of a non-linear system, only a linear system needs to be solved with system matrix depending on the current numerical solution. 

The following theorem states an important property of the scheme.
\begin{theorem}\label{thm:positivity}
  Suppose that the CFL condition
  \begin{equation}\label{eq:cfl}
    \Delta t \leq \min_{1\leq i \leq N} \frac{h_i}{a_i^n},
  \end{equation}
  holds for all $n\in\mathbb{N}_0$, where $a_i^n = |\ddx c(t^n, x_{i-1/2})^- + \ddx c(t^n, x_{i+1/2})^+ |$ in the case $d=1$ and
  \begin{align*}
    a_i^n &= |\ddx c(t^{n}, (x_{j+1/2}, y_k))^+ + \ddx c(t^{n}, (x_{j-1/2}, y_k))^- \\
    &\quad + \ddy c(t^{n}, (x_{j}, y_{k+1/2}))^+ + \ddy c(t^{n}, (x_{j}, y_{k-1/2}))^-| 
  \end{align*}
in the case $d=2$. Then the finite volume scheme given by \eqref{eq:scheme} if $d=1$ and \eqref{eq:scheme2d} if $d=2$ is positivity preserving, i.e., if the initial data satisfies $\rho_h^0\geq 0$ we have $\rho_h^n \geq 0$ for all $n\in\mathbb{N}_0$.
\end{theorem}
\begin{proof}
  We prove this result inductively and assume $\rho_h^n \geq 0$. If $d=1$ we obtain due to $|\ddx c| \geq \ddx c^+, \ddx c^- \geq 0$ and \eqref{eq:cfl} the estimate
  \begin{align}\label{eq:positiverhs}
    \rho_i^n - &\frac{\Delta t}{h_i}\left[ \mathcal{F}_{i+1/2}^{n} - \mathcal{F}_{i-1/2}^{n}\right] \notag
    \\&= \rho_i^n - \frac{\Delta t}{h_i}\left[  \ddx c(t^{n}, x_{i+1/2})^+ \rho_i^{n} -  \ddx c(t^{n}, x_{i+1/2})^- \rho_{i+1}^{n} \right. \notag \\
    &\qquad \left. -  \ddx c(t^{n}, x_{i-1/2})^+ \rho_{i-1}^{n} +  \ddx c(t^{n}, x_{i-1/2})^- \rho_{i}^{n}\right] \notag
    \\ &\geq \rho_i^n \left( 1- \frac{\Delta t}{h_i} \, \left[ \ddx c(t^{n}, x_{i+1/2})^+ + \ddx c(t^{n}, x_{i-1/2})^- \right] \right) \notag\\
    &\geq \rho_i^n \left( 1- \frac{\Delta t}{h_i} a_i^n\right) \geq 0.
  \end{align}
  The scheme \eqref{eq:scheme} can be rewritten as
  \begin{equation}
    \rho_i^{n+1} -  \Delta t^n \Delta_h^\gamma[\rho_h^n, \rho_h^{n+1}] \rvert_{K_i}
   = \rho_i^{n} - \frac {\Delta t}{h_i} \left[ \mathcal{F}_{i+1/2}^{n} - \mathcal{F}_{i-1/2}^{n}\right] \qquad i=1,\dots,N,
 \end{equation}
 which gives rise to the vector form
 \[
   (I - \Delta t A^n) \rho_h^{n+1} = r(\rho_h^n)
   \]
   with $A^n\in\mathbb{R}^{N\times N}$. By the nonnegativity of $\rho_h^n$ also $\hat \rho_{i+1/2}^{n}$ is nonnegative for $i=1,\dots,N$ and thus the matrix $A^n$ has nonpositive diagonal and nonnegative off-diagonal entries. In addition, the entries of each line of $A^n$ add up to zero. For these reasons $I - \Delta t A^n$ is an $M$ matrix and the nonnegativity of $r(\rho_h^n)$ due to \eqref{eq:positiverhs} implies $\rho_h^{n+1}\geq 0$.

   The argument can be transferred to the case $d=2$ in a straight-forward manner, in particular the entries of the right hand side of the linear system are bounded by
  \begin{align}
    \rho_{j,k}^n - &\frac{\Delta t}{h}\left[ \mathcal{F}_{j+1/2,k}^{n} - \mathcal{F}_{j-1/2,k}^{n} + \mathcal{F}_{j,k+1/2}^{n} - \mathcal{F}_{j,k-1/2}^{n}\right] \notag
    \\ &\geq \rho_{j,k}^n \Big( 1- \frac{\Delta t}{h} \, \left[ \ddx c(t^{n}, (x_{j+1/2}, y_k))^+ + \ddx c(t^{n}, (x_{j-1/2}, y_k))^- \right]  \notag \\
    & \quad - \frac{\Delta t}{h} \left( \ddy c(t^{n}, (x_{j}, y_{k+1/2}))^+   + \ddy c(t^{n}, (x_{j}, y_{k-1/2}))^- \right) \Big)  \notag
  \\ &\geq \rho_{j,k}^n \left( 1- \frac{\Delta t}{h} a_{j,k}^n\right) \geq 0
  \end{align}
  using the CFL condition \eqref{eq:cfl}.
\end{proof}

\section{Residual estimates for the 1D scheme}

%Let $x_1, \dots x_N$ denote the centers of the mesh cells $K_1$, \dots, $K_N$. We define $\tilde \rho^n$ as the piecewise linear finite element function with nodes $x_1,\dots,x_N$ interpolating $\rho_1^n,\dots,\rho_N^n$.
To obtain a posteriori estimates for the scheme introduced in Section~\ref{sec:scheme} we employ the stability framework established in Section~\ref{sec:stability}, i.e. we evaluate the stability estimates
%in the residual terms
on  approximate solutions $\tilde \rho $ and $ \tilde c$ that are obtained by suitable reconstructions, to be detailed below, of $\rho_h$ and $c_h$. In this and the following section we provide the required computable bounds of the residual $R_\rho$ occurring
%in the estimates \eqref{eq:classicmainestimate} and \eqref{mainestimate}.
upon inserting $\tilde \rho$ and $ \tilde c$ into the power-law diffusion Keller-Segel system.

To extend the reconstructions and the numerical solution in time we introduce the temporal interpolations
\begin{equation}
\aligned
  \tilde \rho(t,x) = \ell_0^n(t) \tilde \rho^{n+1}(x) + \ell_1^n(t)  \tilde \rho^{n}(x), \qquad t \in[t^{n},t^{n+1}]\\
    \rho_h(t,x) = \ell_0^n(t) \rho_h^{n+1}(x) + \ell_1^n(t)  \rho_h^{n}(x), \qquad t \in[t^{n},t^{n+1}]
   \endaligned
\end{equation}
using the Lagrange polynomials
\[
\ell_0^n(t) = \frac{t-t^n}{\Delta t}, \qquad \ell_1^n(t) = \frac{t^{n+1}-t}{\Delta t}.
\]
% Similarly, we consider the interpolants
% \begin{equation}
%   \rho_i(t) = \ell_0^n(t) \tilde \rho_i^{n+1} + \ell_1^n(t)  \tilde \rho_i^n, \qquad t \in[t^{n},t^{n+1}]
% \end{equation}
% We easily see that
% \begin{equation}\label{eq:laplacetilderho}
%   \frac{1}{|K_i|}\int_{K_i} \Delta \tilde \rho  \, dx = \frac{\rho_{i+1} - 2 \rho_i + \rho_{i-1}}{h^2}.
% \end{equation}
By this definition we have
\begin{equation}\label{eq:ddttilderho}
  \ddt \tilde \rho = \frac{\tilde \rho^{n+1} - \tilde \rho^n}{\Delta t^n}, \qquad t \in (t^{n},t^{n+1}).
\end{equation}
  Next, we introduce the reconstruction $\tilde c$ as the solution to the elliptic equation
  \begin{equation}\label{eq:reconstructionc}
    \tilde c - \Delta \tilde c = \tilde \rho,
  \end{equation}
  which allows us to define the residual
  \begin{equation}
    \tilde R_\rho \coloneqq \ddt \tilde \rho + \nabla \cdot (\tilde \rho \nabla \tilde c) -  \Delta \tilde \rho^\gamma.
  \end{equation}

  In the following, we aim to find a bound for
  \[
    \int_0^T \| \tilde R_\rho \|_{H^{-1}(\rT^1)}^2 \, dt =  \int_0^T \sup_{\substack{\phi \in H^1(\Td)\\ \|\phi \|_{H^1} \leq 1}} \left( \int_{\rT^1}  \tilde R_\rho  \phi \, dx \right)^2 dt.
  \]
  Estimates are first discussed for the non-equidistant scheme in 1D. A generalization to the 2D scheme on Cartesian meshes is discussed in Section \ref{sec:2D}.
  % We analyze the cell averages of the residual over the finite volume mesh, i.e.,
  % \begin{equation}
  %   \tilde R_i \coloneqq \frac{1}{|K_i|} \int_{K_i} \tilde R \, dx.
  %  \end{equation}
  Fixing $t \in[t^{n},t^{n+1}]$ we find that
  \begin{equation}\label{eq:localizedresidual}
    \begin{split}
      \tilde R_\rho  &= \frac{\tilde \rho^{n+1} - \tilde \rho^n}{ \Delta t^n}+ \nabla \cdot (\tilde \rho \nabla \tilde c)
                              -   \Delta \tilde \rho^\gamma 
      \\  &=  \frac{\tilde \rho^{n+1} - \tilde \rho^n}{ \Delta t^n} + \ddx((\ell_0^n \tilde \rho^{n+1} + \ell_1^n \tilde \rho^n) (\ell_0^n \ddx \tilde c^{n+1} + \ell_1^n \ddx \tilde c^{n}))
      \\ &\quad - \Delta \tilde \rho^\gamma + %\frac{\rho_{i+1}^{n+1} - 2 \rho_i^{n+1} + \rho_{i-1}^{n+1}}{h^2}
          \ell_0^n \Delta_h^\gamma[\rho_h^n, \rho_h^{n+1}] + \ell_1^n \Delta_h^\gamma[\rho_h^{n-1}, \rho_h^{n}]  
      \\ & \quad - \frac{\ell_0^n(t)}{h} \left(  d\mathcal{F}^n_h \right) - \frac{\ell_1^n(t)}{h} \left(  d\mathcal{F}_{h}^{n-1}  \right)
     -  \ell_0^n(t) \frac{\rho_h^{n+1} - \rho_h^n}{\Delta t^n} - \ell_1^n(t) \frac{\rho_h^{n} - \rho_h^{n-1}}{\Delta t^n},
    \end{split}
  \end{equation}
  where we have used \eqref{eq:ddttilderho}, the scheme \eqref{eq:scheme}, and the piecewise constant function $ d\mathcal{F}^n_h|_{K_i}:=
  \mathcal{F}_{i+1/2}^{n} - \mathcal{F}_{i-1/2}^{n}$. To estimate the residual we split it as $ \tilde R_\rho  = \tilde R^1 + \tilde R^2 + \tilde R^3$,
such that
  \begin{align}
    \tilde R^1&\coloneqq    \ell_0^n \Delta_h^\gamma[\rho_h^n, \rho_h^{n+1}] + \ell_1^n \Delta_h^\gamma[\rho_h^{n-1}, \rho_h^{n}]  - \Delta \tilde \rho^\gamma, \\
    \tilde R^2&\coloneqq  \frac{\tilde \rho^{n+1} - \tilde \rho^n}{ \Delta t^n} -  \ell_0^n(t) \frac{\rho_h^{n+1} - \rho_h^n}{\Delta t^n} - \ell_1^n(t) \frac{\rho_h^{n} - \rho_h^{n-1}}{\Delta t^{n-1}}, \label{eq:r2} \\
    \tilde R^3&\coloneqq  \ddx((\ell_0^n \tilde \rho^{n+1} + \ell_1^n \tilde \rho^n) (\ell_0^n \ddx \tilde c^{n+1} + \ell_1^n \ddx \tilde c^{n}))  %\nabla \cdot (\tilde \rho \nabla \tilde c)
                  - \frac{\ell_0^n(t)}{h_i}d \mathcal{F}_{h}^{n} - \frac{\ell_1^n(t)}{h_i} d \mathcal{F}_{h}^{n-1}.
  \end{align}

  \subsection{First part of the residual}
We first consider the constituent $\tilde R^1$ of the residual and rewrite it as
\begin{align}
  \tilde R^1 &= \ell_0^n \left( \Delta_h^\gamma[\rho_h^{n}, \rho_h^{n+1}] - \partial_x\left( \gamma (\tilde \rho^{n})^{\gamma -1} \partial_x\tilde \rho^{n+1} \right)\right)+ \ell_1^n \left(\Delta_h^\gamma[\rho_h^{n-1}, \rho_h^{n}]  - \partial_x\left( \gamma (\tilde \rho^{n-1})^{\gamma -1} \partial_x \tilde \rho^{n} \right) \right)  \notag %\Delta \tilde \rho^\gamma \notag
\\  &\quad + \ell_1^n \partial_x \left( (\tilde \rho^{n-1})^{\gamma-1} \partial_x \tilde \rho^{n}  -  (\tilde \rho^n)^{\gamma-1} \tilde \rho_x^{n+1} \right)
 + \partial_x \left( (\tilde \rho^n)^{\gamma-1} \partial_x \tilde \rho^{n+1} - (\tilde \rho)^{\gamma-1} \partial_x \tilde \rho \right)\label{eq:R1largegamma1}.
\end{align}
Note that the latter two terms vanish in the case $\gamma=1$. The sum \eqref{eq:R1largegamma1} is estimated in $H^{-1}(\rT)$ on a term by term basis. To this end we fix $\phi\in H^1(\rT)$ with $\| \phi \|_{H^1(\rT)}\leq 1$.

We estimate the first two terms of \eqref{eq:R1largegamma1} in $H^{-1}(\rT)$ rearranging the summation and integrating by parts. Therefore, we compute %we proceed in analogy to \eqref{eq:R1gamma1} using both the Cauchy-Schwartz and the Hölder inequality and obtain in the case $d=1$ the estimate
  \begin{align}\label{eq:diffresidual1}
    & \int_{\rT^1}\left( \Delta_h^\gamma[\rho_h^{n}, \rho_h^{n+1}]  - \gamma \partial_x\left[(\tilde \rho^{n})^{\gamma-1}\partial_x \tilde \rho^{n+1} \right] \right) \phi\, dx \notag
   % \\&= \sum_{i=1}^N  \frac {\gamma}{h^2} (\hat \rho_{i+1/2}^n)^{\gamma-1} (\rho_{i+1}^{n+1} - \rho_{i}^{n+1}) \left( \int_{K_i} \phi \, dx- \int_{K_{i+1}} \phi \, dx \right) + \int_{K_i}  \gamma (\rho^n)^{\gamma -1}  \rho^{n+1}_x  \phi_x \, dx \notag
    \\ &= \sum_{i=1}^N    \mathcal{D}^{n, n+1}_{i+1/2}  \left( \frac{1}{h_i} \int_{K_i} \phi \, dx- \frac{1}{h_{i+1}}\int_{K_{i+1}} \phi \, dx \right) \notag
         + \sum_{i=1}^N \int_{x_i}^{x_{i+1}} \gamma (\hat \rho_{i+1/2}^n)^{\gamma-1} \partial_x \tilde \rho^{n+1}  \partial_x\phi \, dx
    \\& \quad + \sum_{i=1}^N \int_{x_i}^{x_{i+1}} \gamma \left( (\tilde \rho^n)^{\gamma -1} - (\hat \rho_{i+1/2}^n)^{\gamma-1} \right) \partial_x \tilde \rho^{n+1}  \partial_x \phi \, dx.
  \end{align}
  For the last sum in \eqref{eq:diffresidual1} we estimate using the piecewise linearity of $\tilde \rho$ and Cauchy Schwartz's inequality
  \begin{align}\label{eq:r1part11}
    &\sum_{i=1}^N \int_{x_i}^{x_{i+1}} \gamma \left( (\tilde \rho^n)^{\gamma -1} - (\hat \rho_{i+1/2}^n)^{\gamma-1} \right) \partial_x \tilde \rho^{n+1}  \partial_x \phi \, dx  \notag \\
    &\quad \leq \gamma  \sum_{i=1}^N \| (\tilde \rho^n)^{\gamma -1} - (\hat \rho_{i+1/2}^n)^{\gamma-1} \|_{L^\infty(x_i, x_{i+1})} \|  \partial_x \tilde \rho^{n+1} \|_{L^2(x_i, x_{i+1})} \|\phi \|_{H^1(x_i, x_{i+1})} \notag \\
    &\quad \leq  \gamma  \left( \sum_{i=1}^N \frac{(\rho_{i+1}^{n+1}-\rho_i^{n+1})^2}{d_{i+1/2}} \, \max_{\ell\in \{i, i+1\}}  \left|(\rho_\ell^n)^{\gamma-1} - (\hat \rho_{i+1/2}^n)^{\gamma-1}\right|^2 \right)^{1/2}. 
  \end{align}
  Since $\partial_x\tilde \rho$ is piecewise constant the integral in the remainder in \eqref{eq:diffresidual1} can be computed and we obtain the bound
  \begin{align}
    \sum_{i=1}^N&    \mathcal{D}^{n, n+1}_{i+1/2}  \left( \frac{1}{h_i} \int_{K_i} \phi \, dx- \frac{1}{h_{i+1}}\int_{K_{i+1}} \phi \, dx \right) + \sum_{i=1}^N  \mathcal{D}^{n, n+1}_{i+1/2}  \left(\phi(x_{i+1}) - \phi(x_i) \right) \notag
    \\&= \sum_{i=1}^N \left( \mathcal{D}^{n, n+1}_{i+1/2} -  \mathcal{D}^{n, n+1}_{i-1/2} \right) \frac{1}{h_i} \int_{K_i} (\phi - \phi(x_i)) \, dx \notag
\\ &\leq \sum_{i=1}^N \frac{1}{\sqrt{h_i}} \left| \mathcal{D}^{n, n+1}_{i+1/2} -  \mathcal{D}^{n, n+1}_{i-1/2} \right| \| \phi - \phi(x_i) \|_{L^2(K_i)} \notag
\\ &\leq \left( \sum_{i=1}^N h_i \left( \mathcal{D}^{n, n+1}_{i+1/2} -  \mathcal{D}^{n, n+1}_{i-1/2} \right)^2 \right)^{1/2} \label{eq:r1part12}%\|\phi - \phi(x_i)\|_{L^2(K_i)}
  \end{align}
using Cauchy Schwartz's inequality. The second part of \eqref{eq:R1largegamma1} is bounded by
  \begin{align}\label{eq:r1term3}
    \int_{\rT} \ddx((\tilde \rho^{n-1})^{\gamma-1} \partial_x \tilde \rho^{n}  &-  (\tilde \rho^n)^{\gamma-1} \partial_x \tilde \rho^{n+1}) \phi \, dx \notag \\
     &\leq \left(\sum_{i=1}^N  \|((\rho^{n-1}_i)^{\gamma-1} \tilde \rho_x^{n}  -  (\rho^n_i)^{\gamma-1} \tilde \rho_x^{n+1})\|_{L^2(K_i)}^2 \right)^{1/2}
    % \\ \leq  \left(\sum_{i=1}^N  \| \rho_x^{n+1}\|_{L^2(K_i)}^2\left((\rho^{n}_i)^{\gamma-1} - (\rho^{n-1}_i)^{\gamma-1}\right)^2 \right)^{1/2}
    % + \left(\sum_{i=1}^N  (\rho^{n-1}_i)^{2\gamma-2} \| \tilde \rho_x^{n+1} - \tilde \rho_x^{n} \|^2_{L^2(K_i)} \right)^{1/2}
  \end{align}
  which is readily computable.
  Considering the last term of \eqref{eq:R1largegamma1} we note that the estimate 
  \begin{align*}
    | (\tilde \rho^n)^{\gamma -1} \partial_x \tilde \rho^{n+1} - (\tilde \rho)^{\gamma-1} \partial_x \tilde \rho|
    &\leq \max\{ (\tilde \rho^n)^{\gamma -1}, (\tilde \rho^{n+1})^{\gamma -1}\} |\partial_x \tilde \rho^{n+1} -
    \partial_x \tilde \rho^n|  \\
    &\quad + |(\tilde \rho^n)^{\gamma -1} - (\tilde \rho^{n+1})^{\gamma -1}) |
|\partial_x \tilde \rho^{n+1}|
  \end{align*}
holds and therefore
  \begin{align}
    \int_\rT  &\partial_x \left( (\tilde \rho^n)^{\gamma-1} \partial_x \tilde \rho^{n+1} - (\tilde \rho)^{\gamma-1} \partial_x \tilde \rho \right) \phi \, dx \notag %
    %&\leq \| (\tilde \rho^n)^{\gamma -1} +  (\tilde \rho^{n+1})^{\gamma -1}\|_{L^{\infty}} \| \tilde \rho^{n+1}_x - \tilde \rho^n_x \| \notag
     % \\ &\quad + \|(\tilde \rho^n)^{\gamma -1} - (\tilde \rho^{n+1})^{\gamma -1}) \|_{L^\infty}  \|\tilde \rho^{n+1}_x\|.
    %&\quad
    % -\sum_{i=1}^n \int_{K_i}  \left( (\tilde \rho^n)^{\gamma-1} \tilde \rho_x^{n+1} - (\tilde \rho)^{\gamma-1} \tilde \rho_x \right) \phi_x \, dx
   \\       &\leq \sum_{i=1}^n\| ((\rho^n_i)^{\gamma -1}
     +  (\rho^{n+1}_i)^{\gamma -1})| \partial_x \tilde \rho^{n+1} -
           \partial_x \tilde \rho^n |~ \phi_x  \|_{L^1(K_i)} \notag
    \\&\quad + \sum_{i=1}^n  \| |(\rho^{n}_i)^{\gamma -1}- (\rho^{n+1}_i)^{\gamma -1}|~ |\partial_x \tilde \rho^{n+1} | \, \partial_x \phi\|_{L^1(K_i)} \notag
    \\&\leq \left( \sum_{i=1}^N ((\rho^n_i)^{\gamma -1}
     +  (\rho^{n+1}_i)^{\gamma -1})^2 \|\partial_x \tilde \rho^{n+1} -
   \partial_x \tilde \rho^n\|_{L^2(K_i)}^2\right)^{1/2} \notag
    \\ &\quad + \left( \sum_{i=1}^N  ((\rho^{n}_i)^{\gamma -1}- (\rho^{n+1}_i)^{\gamma -1})^2 \|\partial_x \tilde \rho^{n+1}\|_{L^2(K_i)}^2  \right)^{1/2}. \label{eq:r1term4}
  \end{align}
  In summary, a bound of $\| \tilde R^1 \|_{H^{-1}(\rT)}$ is given by summing the terms \eqref{eq:r1part11}, \eqref{eq:r1part12}, \eqref{eq:r1term3} and \eqref{eq:r1term4}.

\subsection{Second part of the residual}
In this section, a bound of $\tilde R^2$ in $H^{-1}(\rT)$ is derived. Since the Lagrange polynomials satisfy  $\ell_0^n(t)  + \ell_1^n(t) = 1,$ it holds
  \begin{equation}
    \frac{\rho_h^{n+1} - \rho_h^n}{ \Delta t} \rvert_{K_i} = \left( \ell_0^n(t)  + \ell_1^n(t) \right)  \frac{\rho_i^{n+1} - \rho_i^n}{ \Delta t}, 
  \end{equation}
  which allows us to rewrite the second part of the residual as
  \begin{equation}\label{eq:r2rewritten}
    \tilde R^2= \frac{(\tilde \rho^{n+1} -\rho_h^{n+1}) - (\tilde \rho^n - \rho_h^n)}{ \Delta t}  + \ell_1^n(t) \frac{\rho_h^{n+1} - 2\rho_h^{n} + \rho_h^{n-1}}{\Delta t}.
  \end{equation}
  To obtain an estimate for \eqref{eq:r2rewritten} we first derive $H^{-1}(\mathbb{T})$ bounds of $\rho_h - \tilde \rho$. To this end, we use $\phi$ with  $\| \phi\|_{H^1(\rT)}\leq 1$ and compute
  \begin{align*}%\label{eq:rhohtilderhodiff}
    \int_{\rT^1} (\rho_h^n - \tilde \rho^n) \phi \, dx &= \sum_{i=1}^N \int_{x_{i-1/2}}^{x_i}  \left[ \rho_i^n - \left( \frac{x_i -x }{d_{i-1/2}} \rho_{i-1}^n + \frac{x - x_{i-1}}{d_{i-1/2}} \rho_i^n\right) \right]\phi \,dx \notag
    \\ &\quad +  \sum_{i=1}^N \int_{x_i}^{x_{i+1/2}}\left[ \rho_i^n - \left( \frac{x_{i+1} -x }{d_{i+1/2}} \rho_{i}^n + \frac{x - x_{i}}{d_{i+1/2}} \rho_{i+1}^n\right) \right]\phi \,dx.
  \end{align*}
   Using Cauchy-Schwartz's inequality we obtain
   \begin{align}\label{eq:rhohtilderhodiff1}
     \int_{\rT^1} (\rho_h^n - \tilde \rho^n) \phi \, dx &\leq  \sum_{i=1}^N \left(\int_{x_{i-1/2}}^{x_i}  \left| \rho_i^n - \rho_{i-1}^n \right|^2 \, dx\right)^{1/2} \left(\int_{x_{i-1/2}}^{x_i} \phi^2 \,dx\right)^{1/2} \notag
     \\& \quad  +  \sum_{i=1}^N \left(\int_{x_i}^{x_{i+1/2}}\left| \rho_i^n- \rho_{i+1}^n \right|^2 \, dx\right)^{1/2} \left(\int_{x_i}^{x_{i+1/2}} \phi^2 \,dx \right)^{1/2}.
   \end{align}
  Thus, using Cauchy-Schwartz's inequality once more, an $H^{-1}(\rT)$ bound for the first term in \eqref{eq:r2rewritten} is given by
  \begin{equation}\label{eq:rhohtilderhodiff2}
    \left\|\frac{\rho_h^n - \rho_h^{n+1}}{\Delta t} - \frac{\tilde \rho^n - \tilde \rho^{n+1}}{\Delta t}\right\|_{H^{-1}} \leq  \left( \sum_{i=1}^N \frac{h_i+ h_{i-1}}{2}  \left| \frac{\rho_i^n - \rho_i^{n+1}}{\Delta t} - \frac{\rho_{i-1}^n - \rho_{i-1}^{n+1}}{\Delta t} \right|^2 
    \right)^{1/2}.
  \end{equation}
  The second term in \eqref{eq:r2rewritten} is controlled using Cauchy-Schwartz's inequality by
  \begin{align}
    \sum_{i=1}^N \int_{K_i}\ell_1^n(t) \frac{\rho_i^{n+1} - 2\rho_i^{n} + \rho_i^{n-1}}{\Delta t} \, \phi \, dx \leq \left( \sum_{i=1}^N \frac{h_i}{\Delta t^2} \left[\rho_i^{n+1} - 2 \rho_i^n + \rho_i^{n-1} \right]^2\right)^{1/2}.
  \end{align}

\subsection{Third part of the residual}
  Using integration by parts testing the third part of the residual with $\phi\in H^1(\rT)$ yields 
  \begin{align}\label{eq:r3phi}
    \int_{\rT} R^3 \phi \, dx &= - \int_{\rT} (\ell_0^n \tilde \rho^{n+1} + \ell_1^n \tilde \rho^n) (\ell_0^n \partial_x \tilde c^{n+1} + \ell_1^n \partial_x\tilde c^{n}) \partial_x\phi \, dx \notag
    \\ &\quad -  \sum_{i=1}^N\left( \ell_0^n \F^n_{i+1/2} + \ell_1^n \F^{n-1}_{i+1/2}\right) \left( \frac 1 {h_i} \int_{K_i} \phi \, dx- \frac 1 {h_{i+1}} \int_{K_{i+1}} \phi \, dx \right).% \notag
  \end{align}
  We denote by $\Psi_{i+1/2}$ the unique element of $\tilde V_h$ having the value $1$ at $x_{i+1/2}$ and zero at all $x_{j+1/2}$ with $j \not= i$ and remark regarding the last factor in \eqref{eq:r3phi} that
  \begin{equation}\label{eq:diffphi}
   \frac 1 {h_{i+1}} \int_{K_{i+1}} \phi \, dx - \frac 1 {h_i} \int_{K_i} \phi \, dx 
   =- \int_\rT  \phi \partial_x \Psi_{i+1/2} \, dx
    \\= \int_\rT \Psi_{i+1/2} \, \partial_x\phi \, dx.
  \end{equation}
  Therefore \eqref{eq:r3phi} recasts as
  \begin{align}\label{eq:r3final}
  \int_\rT &\left[ \ell_0^n(\tilde \rho^n \partial_x \tilde c^n) + \ell_1^n(\tilde \rho^{n-1}\partial_x\tilde c^{n-1})  - ((\ell_0^n \tilde \rho^{n+1} + \ell_1^n \tilde \rho^n) (\ell_0^n \partial_x \tilde c^{n+1} + \ell_1^n \partial_x \tilde c^{n}))\right] \partial_x\phi \, dx \notag
    \\ - \int_\rT &\left[ \ell_0^n(\tilde \rho^n \partial_x\tilde c^n) + \ell_1^n(\tilde \rho^{n-1}\partial_x\tilde c^{n-1}) \right]  \partial_x\phi \, dx  +  \sum_{i=1}^N\left( \ell_0^n \F^n_{i+1/2} + \ell_1^n \F^{n-1}_{i+1/2}\right)  \int_\rT \Psi_{i+1/2} \, \partial_x\phi \, dx.
    %\frac 1 {\nk{h}} \int_{K_i} \int_x^{x+\nk{h}}  \phi_x(s) \, ds \, dx. \label{eq:r3estimate1}
  \end{align}
  We consider the first integral in \eqref{eq:r3final} and use Cauchy-Schwartz's inequality, the identity
  \[ (\ell_0^n(t) a_1 + \ell_1^n(t) a_2) (\ell_0^n(t) b_1 + \ell_1^n(t) b_2) - \ell_0^n(t) a_1 b_1 - \ell_1^n(t) a_2 b_2  = (a_1 - a_2)(b_1 - b_2) (\tilde t\,^2 - \tilde t) \]
  for $t \in [t^n, t^{n+1}]$ and $\tilde t = (t-t_n)/\Delta t$ and elliptic regularity in \eqref{eq:reconstructionc} to estimate
  \begin{align}\label{eq:r3mixedadvection}
    &\quad \int_\rT \left[  \ell_0^n (\tilde \rho^n \partial_x\tilde c^n - \tilde \rho^{n+1} \partial_x\tilde c^{n+1}) + \ell_1^n  (\tilde \rho^{n-1} \partial_x\tilde c^{n-1}- \tilde \rho^{n} \partial_x \tilde c^{n}) \right] \partial_x \phi \, dx \notag
    \\ & \quad +  \int_\rT \left[ \ell_0^n(\tilde \rho^{n+1}\partial_x \tilde c^{n+1}) + \ell_1^n(\tilde \rho^{n}\partial_x \tilde c^{n})  - ((\ell_0^n \tilde \rho^{n+1} + \ell_1^n \tilde \rho^n) (\ell_0^n \partial_x\tilde c^{n+1} + \ell_1^n \partial_x \tilde c^{n}))\right] \partial_x\phi \, dx \notag
    \\ &\leq  \| \tilde \rho^{n} \partial_x\tilde c^{n} - \tilde \rho^{n+1} \partial_x\tilde c^{n+1} \|_{L^2(\rT)} + \| \tilde \rho^{n-1} \partial_x\tilde c^{n-1} - \tilde \rho^{n} \partial_x\tilde c^{n} \|_{L^2(\rT)} \notag 
    \\& \quad + \| (\tilde \rho^{n+1} - \tilde \rho^n) (\partial_x\tilde c^{n+1} - \partial_x\tilde c^n) \|_{L^2(\rT)} \notag
    \\ &\leq \| \tilde \rho^{n} (\partial_x\tilde c^{n}  - \partial_x\tilde c^{n+1})  \notag \\
    &\quad + \partial_x \tilde c^{n+1}(\tilde \rho^{n} - \tilde \rho^{n+1}) \|_{L^2(\rT)} +  \| \tilde \rho^{n-1} (\partial_x\tilde c^{n-1}  - \partial_x\tilde c^{n})  +  \partial_x\tilde c^{n}(\tilde \rho^{n-1} - \tilde \rho^{n}) \|_{L^2(\rT)} \notag
    \\ &\quad + \| (\tilde \rho^{n+1} - \tilde \rho^n) (\partial_x\tilde c^{n+1} - \partial_x\tilde c^n) \|_{L^2(\rT)} \notag
    \\ &\leq \| \tilde \rho^n \|_{L^\infty(\rT)} \| \partial_x\tilde c^{n} - \partial_x\tilde c^{n+1} \|_{L^2(\rT)} + \| \partial_x\tilde c^{n+1}\|_{L^2(\rT)} \| \tilde \rho^{n} - \tilde \rho^{n+1}\|_{L^\infty(\rT)} \notag \\
    & \quad + \| \tilde \rho^{n-1} \|_{L^\infty(\rT)} \| \partial_x\tilde c^{n-1}  - \partial_x\tilde c^{n} \|_{L^2(\rT)} + \|\partial_x \tilde c^{n}\|_{L^2(\rT)} \| \tilde \rho^{n-1} - \tilde \rho^{n}\|_{L^\infty(\rT)} \notag
    \\ &\quad  + \| \tilde \rho^{n+1} - \tilde \rho^n\|_{L^\infty(\rT)}   \|\partial_x\tilde c^{n+1} - \partial_x\tilde c^n \|_{L^2(\rT)} \notag
    \\ &\leq  \left( \| \tilde \rho^n \|_{L^\infty(\rT)} + \| \tilde \rho^{n+1} \|_{L^\infty(\rT)} + \| \tilde \rho^{n+1} - \tilde \rho^n\|_{L^\infty(\rT)} \right) \| \tilde \rho^{n}  - \tilde \rho^{n+1} \|_{L^\infty(\rT)}    \notag
    \\ & \quad  + \left( \| \tilde \rho^{n-1} \|_{L^\infty(\rT)} + \| \tilde \rho^{n}\|_{L^\infty(\rT)}  \right)  \| \tilde \rho^{n-1} - \tilde \rho^{n}\|_{L^\infty(\rT)}.
  \end{align}
 In the last inequality the embedding $L^\infty \hookrightarrow L^2$ has been used. We next derive a bound for the remaining terms of \eqref{eq:r3final}. To this end, we neglect the time and remark that $\Psi_{i+1/2}$ for all $i=1,\dots,N$ form a partition of unity, which allows us to write and estimate the remainder as
  \begin{align}\label{eq:r3flux}
    \sum_{i=1}^N \int_{K_i \cup K_{i+1}} (\F_{i+1/2} - \tilde \rho \partial_x\tilde c)& \Psi_{i+1/2} \, \partial_x\phi \, dx \notag \\
    &\leq \sum_{i=1}^N \| (\F_{i+1/2} - \tilde \rho \partial_x\tilde c)\Psi_{i+1/2} \|_{L^2(K_i \cup K_{i+1})} \|\partial_x \phi\|_{L^2(K_i \cup K_{i+1})} \notag
    \\ &\leq 2 \left( \sum_{i=1}^N \| \F_{i+1/2} - \tilde \rho \partial_x \tilde c \|_{L^2(K_i \cup K_{i+1})}^2\right)^{1/2}.
    % \int_\Td &\left[ \ell_0^n(\tilde \rho^n \tilde c^n_x) + \ell_1^n(\tilde \rho^{n-1}\tilde c^{n-1}_x) \right]  \Psi_{i+1/2} \, \phi_x \, dx  +  \sum_{i=1}^N\left( \ell_0^n \F^n_{i+1/2} + \ell_1^n \F^{n-1}_{i+1/2}\right)  \int_\rT \Psi_{i+1/2} \, \phi_x \, dx
  \end{align}
Using Hölder's inequality we further estimate 
\begin{align}\label{eq:localfluxerr}
  \| \F_{i+1/2} - \tilde \rho \partial_x\tilde c\|_{L^2(K_i)}^2  &\leq  \| g(\rho_i, \rho_{i+1}) \partial_x  c_h - \tilde \rho \partial_x \tilde c(x_{i+1/2})\|^2_{L^2(K_i)}
  \notag
\\ &\leq  \|\partial_x  c_{h}\|^2_{L^2(K_i)} \| g(\rho_i, \rho_{i+1}) - \tilde \rho \|^2_{L^\infty(K_i)}  + \| \tilde \rho \|_{L^{\infty}(K_i)}^2 \| \partial_x (c_h - \tilde c) \|_{L^2(K_i)}^2 \notag
  \\  &\leq  \|\partial_x  c_{h}\|^2_{L^2(K_i)} \max_{\ell \in \{ i, i+1\}} \| \rho_\ell - \tilde \rho \|^2_{L^\infty(K_i)} + \| \tilde \rho \|_{L^{\infty}(K_i)}^2 \| \partial_x (c_h - \tilde c) \|_{L^2(K_i)}^2, 
\end{align}
where $g(\rho_i, \rho_{i+1})=\rho_i$ if $\partial_x c_h(x_{i+1/2})>0$ and $g(\rho_i, \rho_{i+1})=\rho_{i+1}$ otherwise. Similarly, the following bound holds
\begin{align}\label{eq:localfluxerr2}
  \| \F_{i+1/2} - \tilde \rho \partial_x \tilde c\|_{L^2(K_{i+1})}^2 &\leq \|\partial_x  c_{h}\|^2_{L^2(K_{i+1})} \max_{\ell \in \{ i, i+1\}} \| \rho_\ell - \tilde \rho \|^2_{L^\infty(K_{i+1})} \notag \\
  & \quad + \| \tilde \rho \|_{L^{\infty}(K_{i+1})}^2 \| \partial_x (c_h - \tilde c) \|_{L^2(K_{i+1})}^2
\end{align}
and thus combining \eqref{eq:r3mixedadvection}, \eqref{eq:r3flux}, \eqref{eq:localfluxerr}, and \eqref{eq:localfluxerr2}
with the a posteriori bound 
\[ \| c_h^n - \tilde c(t^n) \|_{H^1(\rT^1)}^2
 \leq \sum_{i=1}^N d_{i+1/2}^2 \| c_h^n - \tilde \rho^n\|_{L^2(x_i,x_{i+1})}^2 
 + \sum_{i=1}^N h_i (\partial_x c_h^n(x_i^+) - 
\partial_x c_h^n(x_i^-))^2, \]
see \cite{Verfurth_book},
we obtain a computable bound for the second part of \eqref{eq:r3final}.

\section{Residual estimates for the 2D scheme}\label{sec:2D}
In this section we discuss an adaptation of the above residual estimates for the 2D scheme~\eqref{eq:scheme2d}.

\subsection{First part of the residual} Proceeding in analogy to \eqref{eq:localizedresidual} the first part of the residual takes the form
\begin{align}
\tilde R^1 &=  \ell_0^n \Delta_{x,h}^\gamma[\rho_h^n, \rho_h^{n+1}] + \ell_1^n \Delta_{x,h}^\gamma[\rho_h^{n-1}, \rho_h^{n}]  - \Delta_x \tilde \rho^\gamma  \notag
\\ &\quad +   \ell_0^n \Delta_{y,h}^\gamma[\rho_h^n, \rho_h^{n+1}] + \ell_1^n \Delta_{y,h}^\gamma[\rho_h^{n-1}, \rho_h^{n}]  - \Delta_y \tilde \rho^\gamma .
\end{align}
The constituents with respect to the $x$- and to the $y-$derivative each allow for a reformulation analogous to \eqref{eq:R1largegamma1}. The terms accounting for the time discretization are estimated as in the one-dimensional case, i.e., the bounds \eqref{eq:r1term3} and \eqref{eq:r1term4} hold. The terms accounting for the approximation of the nonlinear diffusion terms are also estimated analogously as in the one-dimensional scheme: considering the $x$-derivative we obtain
  \begin{align}
    &\int_{\rT^2} \left( \Delta_{x,h}^\gamma[\rho_h^n, \rho_h^{n+1}] - \gamma \partial_x \left[(\tilde \rho^{n})^{\gamma-1}\partial_x\tilde \rho^{n+1} \right] \right) \phi\, dx \notag
    \\ &= \sum_{j, k}   \mathcal{D}_{j+1/2,k}^{n,n+1}\left( \frac{1}{h}\int_{K_{j,k}} \phi \, dx  - \frac{1}{h}\int_{K_{j+1,k}} \phi \, dx \right) \notag \\
    &\quad +  \sum_{j, k} \int_{y_{k-1/2}}^{y_{k+1/2}} \int_{x_j}^{x_{j+1}} \gamma (\hat \rho_{j+1/2,k}^n)^{\gamma-1} \partial_x \tilde \rho^{n+1} \partial_x \phi \, dx dy \notag
    \\ &\quad + \sum_{j, k} \int_{y_{k-1/2}}^{y_{k+1/2}} \int_{x_j}^{x_{j+1}} \gamma \left( (\tilde \rho^n)^{\gamma -1} - (\hat \rho_{j+1/2,k}^n)^{\gamma-1} \right) \partial_x \tilde \rho^{n+1}  \partial_x\phi \, dx \, dy. \label{eq:flux2de1}
  \end{align}
  The last term occurring in \eqref{eq:flux2de1} is bounded by
  \begin{align}
    \sum_{j, k} &\int_{y_{k-1/2}}^{y_{k+1/2}} \int_{x_j}^{x_{j+1}} \gamma \left( (\tilde \rho^n)^{\gamma -1} - (\hat \rho_{j+1/2, k}^n)^{\gamma-1} \right) \partial_x \tilde \rho^{n+1}  \partial_x\phi \, dx \, dy \notag
    \\  &\leq \sum_{j, k} \gamma  \| (\tilde \rho^n)^{\gamma -1} - (\hat \rho_{j+1/2, k}^n)^{\gamma-1} \|_{L^\infty(\tilde K_{j,k})} \|  \partial_x \tilde \rho^{n+1} \|_{L^2(\tilde K_{j,k})} \|\phi \|_{H^1(\tilde K_{j,k})} \notag
    \\ &\leq  \gamma \left( \sum_{j,k}  \| (\tilde \rho^n)^{\gamma -1} - (\hat \rho_{j+1/2, k}^n)^{\gamma-1} \|_{L^\infty(\tilde K_{j,k})}^2  \| \partial_x \tilde \rho^{n+1} \|_{L^2(\tilde K_{j,k})}^2
         \right)^{1/2},  
  \end{align}
  where $\tilde K_{j,k} = [x_{j}, x_{j+1}]\times[y_{k-1/2}, y_{k+1/2}]$ refers to a shifted version of the cell $K_{j,k}$ in $x-$direction.
  For estimating the remainder of \eqref{eq:flux2de1}, we define
  $\chi_j=\chi_j(x)$ as $1/h$ in $[x_{j-1/2},x_{j+1/2}]$ and zero outside of this interval.
  Then, the first two terms in  \eqref{eq:flux2de1} equal 
  \begin{align}
    \int_{\rT^2} \sum_{j, k} &\left(    \chi_k(y) \Psi_{j+1/2}(x) \mathcal{D}^{n,n+1}_{j+1/2,k}
    - \chi_k(y)\chi_{j+1/2}(x) \mathcal{D}^{n,n+1}_{j+1/2,k}\right)  \partial_x \phi \notag \\
    &\leq \frac{1}{2}\sum_{j,k} \left( h^2 |\mathcal{D}^{n,n+1}_{j+1/2,k} - \mathcal{D}^{n,n+1}_{j-1/2,k}|^2\right)^{1/2} 
    \| \partial_x \phi\|_{L^2(K_{jk})} \notag \\
   & \leq \frac{h}{2} \left( \sum_{j,k} |\mathcal{D}^{n,n+1}_{j+1/2,k} - \mathcal{D}^{n,n+1}_{j-1/2,k}|^2\right)^{1/2}.
    \end{align}
%   \begin{align}
%     \sum_{j, k}   &\mathcal{D}_{j+1/2,k}^{n,n+1}\left( \frac{1}{h} \int_{K_{j,k}} \phi \, dx  - \frac{1}{h}\int_{K_{j+1,k}} \phi \, dx \right) +  \sum_{j, k} \int_{y_{k-1/2}}^{y_{k+1/2}} \sum_{\ell=1}^m \mathcal{D}_{j+1/2,\ell}^{n,n+1} \Psi_\ell(y) (\phi(x_{i+1}, y) - \phi(x_i, y)) \, dy \notag
%     \\ &\leq \sum_{j,k} \left( \mathcal{D}_{j+1/2,k}^{n,n+1} -  \mathcal{D}_{j-1/2,k}^{n,n+1}\right) \frac{1}{h} \int_{x_{j-1/2}}^{x_{j+1/2}} \int_{y_{k-1/2}}^{y_{k+1/2}} (\phi(x, y) - \phi(x_i,y)) \, dx dy \notag
%     \\ &\quad - \sum_{j, k, \ell} \int_{y_{k-1/2}}^{y_{k+1/2}} \left(\mathcal{D}_{j+1/2,k}^{n,n+1} - \mathcal{D}_{j+1/2,\ell}^{n,n+1}\Psi_\ell(y) \right) (\phi(x_{i+1}, y) - \phi(x_i, y)) \, dy \notag
%     \\ &\leq h \left( \sum_{j,k} \left(\mathcal{D}_{j+1/2,k}^{n,n+1} -  \mathcal{D}_{j-1/2,k}^{n,n+1}  \right)^2 \right)^{1/2}  + \frac{h}2 \left( \sum_{j,k} \max_{\ell \in \{k-1, k+1\}}  \left|\mathcal{D}_{j+1/2,k}^{n,n+1} -  \mathcal{D}_{j+1/2,\ell}^{n,n+1}  \right|^2 \right)^{1/2}
%   \end{align}
 % with $\Psi_\ell$ for $1\leq \ell \leq m$ denoting the linear basis function with respect to the points $y_1, \dots, y_m$.

\subsection{Second part of the residual} The second part of the residual takes also in case of the 2D scheme the form  \eqref{eq:r2rewritten}. To obtain a computable bound we estimate $\rho_h - \tilde \rho$ on the Cartesian grid. To this end we decompose all mesh cells as
  \begin{align*}
    K_{j,k} &= [x_{j-1/2}, x_j] \times [y_{k-1/2}, y_k] \cup [x_{j}, x_{j+1/2}] \times [y_{k-1/2}, y_k]
    \\ & \quad \cup [x_{j-1/2}, x_j] \times [y_{k}, y_{k+1/2}] \cup [x_{j}, x_{j+1/2}] \times [y_{k}, y_{k+1/2}]
         \\ & \eqqcolon  K_{j,k}^\text{SW} \cup K_{j,k}^\text{SE} \cup K_{j,k}^\text{NW} \cup K_{j,k}^\text{NE} \label{eq:partition}
  \end{align*}
then, neglecting the time index $n$, the piecewise bilinear function $\tilde \rho$ takes the form 
\begin{align*}
  \tilde \rho(x,y) \rvert_{K_{j,k}^\text{SW}} &= \left( 1 + \frac{x-x_j}{h} + \frac{y-y_k}{h}  \right) \, \rho_{j,k} - \frac{x-x_j}{h} \rho_{j-1,k} - \frac{y-y_k}{h} \rho_{j,k-1},
 \\ \tilde \rho(x,y) \rvert_{K_{j,k}^\text{SE}} &= \left( 1 - \frac{x-x_j}{h} + \frac{y-y_k}{h}  \right) \, \rho_{j,k} + \frac{x-x_j}{h} \rho_{j+1,k} - \frac{y-y_k}{h} \rho_{j,k-1},
 \\ \tilde \rho(x,y) \rvert_{K_{j,k}^\text{NW}} &= \left( 1 + \frac{x-x_j}{h} - \frac{y-y_k}{h}  \right) \, \rho_{j,k} - \frac{x-x_j}{h} \rho_{j-1,k} + \frac{y-y_k}{h} \rho_{j,k+1},
 \\ \tilde \rho(x,y) \rvert_{K_{j,k}^\text{NE}} &= \left( 1 - \frac{x-x_j}{h} - \frac{y-y_k}{h}  \right) \, \rho_{j,k} + \frac{x-x_j}{h} \rho_{j+1,k} + \frac{y-y_k}{h} \rho_{j,k+1}.
\end{align*}
Assuming a constant test function $\phi\equiv \phi_{j,k}$ in $K_{j,k}$ we note that
\begin{align*}
  \int_{K_{j,k}^\text{SW}} (\rho_h - \tilde \rho) \phi_{j,k} \, dx &= \frac{h}{2} \int_{x_{j-1/2}}^{x_j} \frac{x - x_j}{h} (\rho_{j-1,k} - \rho_{j,k})  \phi_{j,k} \, dx \\
  &\quad + \frac{h}{2} \int_{y_{k-1/2}}^{y_k} \frac{y - y_k}{h} (\rho_{j,k-1} - \rho_{j,k})  \phi_{j,k} \, dy
  \\ &= \frac{h^2}{16} \left[\rho_{j,k} - \rho_{j-1,k}\right]  \phi_{j,k}  + \frac{h^2}{16} \left[\rho_{j,k} - \rho_{j,k-1}\right]  \phi_{j,k}.
\end{align*}
Similar identities hold when integrating over $ K_{j,k}^\text{SE}$, $K_{j,k}^\text{NW}$ and $K_{j,k}^\text{NE}$ giving rise to
\begin{align}
  \sum_{j,k} \int_{K_{j,k}} &(\rho_h - \tilde \rho) \phi_{j,k} \, dx = \sum_{j,k} \frac{h^2}{32} \left[-\rho_{j-1,k} - \rho_{j, k-1} + 4 \rho_{j,k} - \rho_{j,k-1} - \rho_{j-1,k} \right]  \phi_{j,k} \notag
                                                                         \\ &\leq \sum_{j,k} \left( \frac{h^2}{1024} \left[-\rho_{j-1,k} - \rho_{j, k-1} + 4 \rho_{j,k} - \rho_{j,k-1} - \rho_{j-1,k} \right]^2 \right)^{1/2} \coloneqq h \, P[\rho_h].
\end{align}
Focussing again on the south west part of the mesh cells we estimate the remainder term 
\begin{align}
  \sum_{j,k} \int_{K_{j,k}^\text{SW}} &(\rho_h -\tilde  \rho) (\phi - \phi_{j,k}) \, dx \notag
  \\ &\leq \sum_{j,k} \left\| \frac{x - x_j}{h} (\rho_{j-1,k} - \rho_{j,k})  +  \frac{y - y_k}{h} (\rho_{j,k-1} - \rho_{j,k}) \right\|_{L^2(K_{j,k})}  \|\phi - \phi_{j,k} \|_{L^2(K_{j,k})} \notag
  \\ &\leq \sum_{j,k} h^{2} \left(|\rho_{j-1,k} - \rho_{j,k}| + |\rho_{j,k-1} - \rho_{j,k}| \right) \| \phi \|_{H^1(K_{j,k})} \notag
  \\ &\leq \left( \sum_{j,k} h^{4} \left(|\rho_{j-1,k} - \rho_{j,k}| + |\rho_{j,k-1} - \rho_{j,k}| \right)^2 \right)^{1/2} \eqqcolon h^2 \, Q_\text{SW}[\rho_h].
\end{align}
We note that similar estimates hold, when focussing on any of the other subcells in \eqref{eq:partition}.
Eventually the first term in \eqref{eq:r2rewritten} can be bounded in $H^{-1}(\rT^2)$ by
  \begin{multline}
       \frac{h}{\Delta t} P[\rho_h^{n+1}-\rho_h^n] + \frac{h^2}{\Delta t} \left( Q_\text{SW}[\rho_h^{n+1}-\rho_h^n] + Q_\text{SE}[\rho_h^{n+1}-\rho_h^n] \right.\\ \left.
    + Q_\text{NW}[\rho_h^{n+1}-\rho_h^n] + Q_\text{NE}[\rho_h^{n+1}-\rho_h^n]\right).
      \end{multline}

      \subsection{Third part of the residual} Following the computation \eqref{eq:localizedresidual} in case of the 2D scheme we obtain as  third part of the residual 
\begin{align}\label{eq:r32d}
\tilde R^3 &= \partial_x ((\ell_0^n \tilde \rho^{n+1} + \ell_1^n \tilde \rho^n) (\ell_0^n \partial_x \tilde c^{n+1} + \ell_1^n \partial_x \tilde c^{n}))       - \frac{\ell_0^n}{h}  d_x\mathcal{F}_{h}^{n}  \notag
               - \frac{\ell_1^n}{h} d_x \mathcal{F}_{h}^{n-1} \\
  &\quad +\partial_y ((\ell_0^n \tilde \rho^{n+1} + \ell_1^n \tilde \rho^n) (\ell_0^n  \partial_y \tilde c^{n+1} + \ell_1^n \partial_y \tilde c^{n})) - \frac{\ell_0^n}{h} d_y  \mathcal{F}_{h}^{n}
               - \frac{\ell_1^n}{h} d_y  \mathcal{F}_{h}^{n-1} .
\end{align}
where $d_x \mathcal{F}_{h}^{n-1}$ and $d_y \mathcal{F}_{h}^{n-1}$ are defined analogously to the 1D case.
After testing with $\phi\in H^{-1}(\rT^2)$ the residual part \eqref{eq:r32d} can be split into components accounting for the $x$- and the $y$- derivative respectively, each of which can be rewritten in analogy to \eqref{eq:r3phi}. For the first part accounting for the mixed time instances in the advection term the bound \eqref{eq:r3mixedadvection} holds in 2D analogously. A computation analogue to \eqref{eq:diffphi} shows that
  \begin{equation}
 \frac 1 {h} \left(\int_{K_{j+1,k}} \phi \, dx- \int_{K_{j,k}} \phi \, dx\right) = \int_{y_{k-1/2}}^{y_{k+1/2}} \int_{x_{j-1/2}}^{x_{j+3/2}} \Psi_{j+1/2}(x) \, \ddx \phi \, dx \, dy.
\end{equation}
We thus estimate the second part of the 2D equivalent of \eqref{eq:r3phi} corresponding to the $x$-derivative by 
\begin{align}
  \sum_{j,k} \int_{y_{k-1/2}}^{y_{k+1/2}} \int_{x_{j-1/2}}^{x_{j+3/2}} &\left( \F_{j+1/2, k} - \tilde \rho \ddx \tilde c\right) \Psi_{j+1/2}(x) \ddx \phi \, dx \, dy \notag
   \\ &\leq 2 \left( \sum_{j,k} \| \F_{j+1/2,k} - \tilde \rho \ddx \tilde c \|_{L^2(K_{j,k} \cup K_{j+1,k})}^2\right)^{1/2}.
\end{align}
Replicating the computations in \eqref{eq:localfluxerr} and \eqref{eq:localfluxerr2}, we obtain 
\begin{align}\label{eq:localfluxerr2d}
  \| \F_{j+1/2, k} - \tilde \rho \tilde c_x&\|_{L^2(K_{j,k} \cup K_{j+1,k})}^2  
   \leq  \|\partial_x  c_{h}\|^2_{L^2(K_{j,k} \cup K_{j+1,k})} \max_{\ell \in \{ j, j+1\}} \| \rho_{\ell,k} - \tilde \rho \|^2_{L^\infty(K_{j,k} \cup K_{j+1,k})} \notag \\
   &\quad +  \| \tilde \rho \|_{L^{\infty}(K_{j,k} \cup K_{j+1,k})}^2 h^2 \| \partial_x c_h - \partial_x c_h (x_{j+1/2},y_k) \|_{L^\infty(K_{j,k} \cup K_{j+1,k})}^2\notag\\
      &\quad + \| \tilde \rho \|_{L^{\infty}(K_{j,k} \cup K_{j+1,k})}^2 \| \partial_x (c_h - \tilde c) \|_{L^2(K_{j,k} \cup K_{j+1,k})}^2, 
\end{align}
where the last term is controlled employing the a posteriori bound 
\begin{equation}\label{eq:ellipticestimate}
  \| c_h^n - \tilde c(t^n) \|_{H^1(\rT^2)}^2
 \leq \sum_{T \in \cT_h}^N h^2 \| c_h^n - \tilde \rho^n\|_{L^2(T)}^2 
 + \sum_{E_{T|V} \in \mathcal{E}_h}  h \, |E_{T|V}| \, \left((\nabla c_h^n\rvert_T - 
   \nabla c_h^n\rvert_V) \eta_{T|V}\right)^2
\end{equation}
from \cite{Verfurth_book} with $\mathcal{E}_h$ denoting the set of edges within $\cT_h$, $E_{T|V}$ the edge between $T$ and $V$ and $\eta_{T|V}$ the outer normal vector of $T$ towards $V$.

% the following local approximation error of the numerical fluxes
% \begin{align}\label{eq:localfluxerr2d}
%   \| \F_{j+1/2, k} - \tilde \rho \tilde c_x\|_{L^2(K_{j,k} \cup K_{j+1,k})}^2 &\leq h^2 \,  \tilde c_{x}(x_{j+1/2}, y_k)^2  \max_{\ell \in \{ j, j+1\}} \| \rho_{\ell, k} - \tilde \rho \|^2_{L^\infty(K_{j,k} \cup K_{j+1,k})} \notag
% \\ &\quad   + h^2 \left( \| \tilde \rho \|_{L^{\infty}(K_{j,k})}^2   \| \tilde \rho\|_{L^2(K_{j,k})}^2 + \| \tilde \rho \|_{L^{\infty}(K_{j+1,k})}^2   \| \tilde \rho\|_{L^2(K_{j+1,k})}^2 \right).
% \end{align}

\section{Numerical experiments}
To illustrate the behavior of the residual, and thus of the error estimator, we present three numerical experiments using the 2D scheme on the domain $\mathbb{T}^2$. We consider scenarios which differ with respect to the cell migration and set $\gamma=1$ in the first, $\gamma=1.5$ in the second and $\gamma=2$ in the third experiment.
The initial cell density is set to a modified Gaussian function centered at the point $\left(\frac 12, \frac 12\right)$ in detail given by
\begin{equation}
  \rho_0(x, y) = 1.3 \, \sin(\pi x) \, \sin(\pi y) \, e^{-25 (x-\frac 12 )^2 -25 (y - \frac 12)^2}.
\end{equation}
Numerical solutions are computed up to the final time $T=5 \times 10^{-3}$ using Cartesian meshes consisting of $n \times n$ cells and the uniform mesh resolution dependent step sizes $\Delta t = \frac{T}{n}$. During the simulation we verify that by this choice the CFL condition in Theorem~\ref{thm:positivity} remains satisfied and thus the numerical approximation preserves its initial positivity. For efficiency we employ mass lumping in \eqref{eq:cfem} to compute the chemical attractant $c_h$, which gives rise to linear systems that can be solved using a fast Fourier transform of $\rho_h$. We note that also using mass lumping an a posteriori estimate of the form \eqref{eq:ellipticestimate} holds, cf. \cite{Gupta2017}.

\begin{figure}
 \centering
 \includegraphics[scale=0.95]{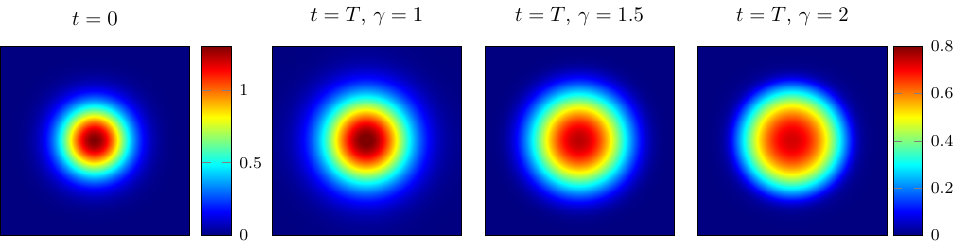}
\caption{Cell density $\rho$ over $\rT^2$ at initial and final time $T=5\times 10^{-3}$ for Experiments 1--3 computed using the 2D scheme on a mesh of $200 \times 200$ cells.}\label{fig:densities}
\end{figure}
The qualitative behavior of the numerical solution with respect to the choice of $\gamma$ is shown in Figure~\ref{fig:densities}. As $\gamma$ increases lower cell concentrations spread slower over the domain, while the diffusion of higher concentrations occurs with higher speed.  

To verify the conditions \eqref{eq:classicstabcondition} and \eqref{eq:stabilitycondition} that ensure validity of our a posteriori error estimate we compute numerical estimates of the embedding constants occurring in \eqref{eq:a} taking into account the computational domain and our choice of $\gamma$. Using the result in \cite[Lemma 2.3]{cai2012fitzhnagum} we obtain $C_S \approx 2.1358$. Combining the embeddings from $H^2$ to $W^{1,6}$ and from $W^{1,6}$ to $L^\infty$ employing \cite[Theorem 3.4]{mizuguchi2017estimsobol} we compute $C_S^\prime \approx 7.6112$. Lastly, by using the embedding $H^2 \hookrightarrow W^{1,\frac{2 \gamma + 2}{\gamma -1}}$ and \cite[Theorem 3.3]{mizuguchi2017estimsobol} we obtain $\tilde C_S^2\approx 5.2494$ in case of experiment 2 and $\tilde C_S^3\approx 3.9228$ in case of experiment 3. 
\begin{figure}
  \centering
  \includegraphics[scale=0.94]{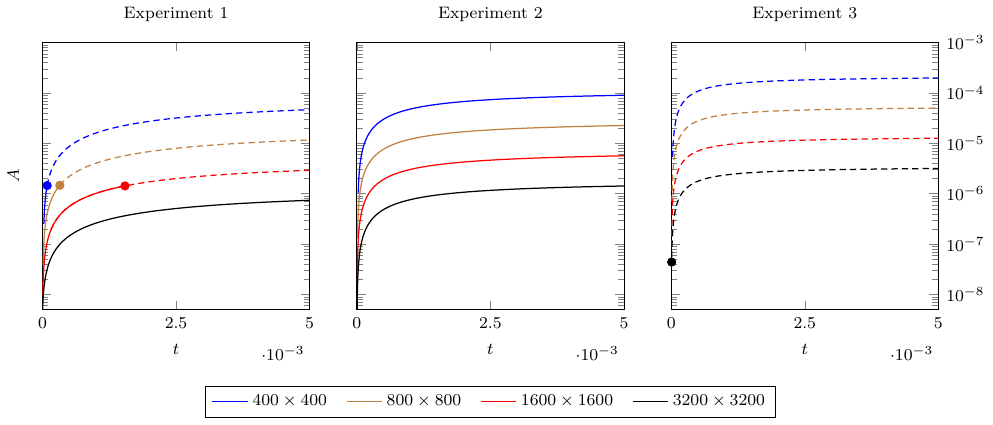}
\caption{The residual $A$ over the time interval $[0,0.005]$ in Experiment 1 ($\gamma=1$), 2 ($\gamma=1.5$) and 3 ($\gamma=2$) for the 2D scheme using the mesh resolutions $400 \times 400$, $800 \times 800$, $1600 \times 1600$ and $3200 \times 3200$. Dashed lines indicate the regime, in which the stability condition \eqref{eq:classicstabcondition} for Experiment~1 and  \eqref{eq:stabilitycondition} for Experiments~2--3 is not satisfied. Note that our residuals are valid for longer times in Experiment 2 than in Experiment 1 due to different exponents in the respective stability conditions.}\label{fig:residuals}
\end{figure}

For all three experiments Figure~\ref{fig:residuals} presents the total residual
\begin{equation}\label{eq:totalresidual}
  A(t)= z_1(0)  + \int_0^{T} \| R_\rho\|_{H^{-1}}^2  dt                                                          
\end{equation}
over the considered time interval for various mesh resolutions ranging from $n=400$ to $n=3200$. The total residual is clearly reduced as finer meshes are considered, whereas an increase with the choice of $\gamma$ is observed. In Experiment~1 the stability condition \eqref{eq:classicstabcondition} remains satisfied over the full time interval for mesh resolution $n=3200$. Also shorter time computations using coarser meshes of resolution $n=400$, $n=800$ or $n=1600$ can be conducted without violating the stability conditions. The condition \eqref{eq:stabilitycondition} allows for significantly larger total residuals in Experiment~2 than in Experiment~1 as it remains satisfied over the full time interval for all shown mesh resolutions. In case of $\gamma=2$ in Experiment 3 even finer mesh resolutions than the considered ones are necessary to satisfy condition \eqref{eq:stabilitycondition}. This is due to the larger exponent $\frac{1}{\beta}$, which in this experiment requires a total residual of approximate magnitude $10^{-7}$.

\begin{table}\small
  \centering
  \caption{Mesh convergence for the total and the restricted residuals given by \eqref{eq:totalresidual} and \eqref{eq:restricted} in case of the classical Keller--Segel model in Experiment 1.}\label{tab:exp1}
  \begin{tabular}{rlrlrlrlr}
    \toprule
    mesh & $A^1$ & EOC & $A^2$ & EOC & $A^3$ & EOC & $A$ & EOC\\ \midrule
    $100$ & \(6.085 \times 10^{-4}\) & & \(1.766 \times 10^{-4}\) & &
\(3.629 \times 10^{-6}\) & & \(7.888 \times 10^{-4}\) & \\
$200$ & \(1.538 \times 10^{-4}\) & 1.98 & \(3.326 \times 10^{-5}\) & 2.41
& \(9.207 \times 10^{-7}\) & 1.98 & \(1.880 \times 10^{-4}\) & 2.07 \\
$400$ & \(3.860 \times 10^{-5}\) & 1.99 & \(7.656 \times 10^{-6}\) & 2.12
& \(2.314 \times 10^{-7}\) & 1.99 & \(4.648 \times 10^{-5}\) & 2.02 \\
$800$ & \(9.680 \times 10^{-6}\) & 2.00 & \(1.884 \times 10^{-6}\) & 2.02
& \(5.804 \times 10^{-8}\) & 2.00 & \(1.162 \times 10^{-5}\) & 2.00 \\
$1600$ & \(2.424 \times 10^{-6}\) & 2.00 & \(4.762 \times 10^{-7}\) & 1.98
& \(1.453 \times 10^{-8}\) & 2.00 & \(2.915 \times 10^{-6}\) & 2.00 \\
$3200$ & \(6.065 \times 10^{-7}\) & 2.00 & \(1.298 \times 10^{-7}\) & 1.88
& \(3.636 \times 10^{-9}\) & 2.00 & \(7.399 \times 10^{-7}\) & 1.98\\ \bottomrule
  \end{tabular}
\end{table}

\begin{table}\small
  \centering
  \caption{Mesh convergence for the total and the restricted residuals given by \eqref{eq:totalresidual} and \eqref{eq:restricted} in case of the power-law Keller--Segel model in Experiment 2.}\label{tab:exp2}
  \begin{tabular}{rlrlrlrlr}
    \toprule
    mesh & $A^1$ & EOC & $A^2$ & EOC & $A^3$ & EOC & $A$ & EOC\\ \midrule
    $100$ & \(7.794 \times 10^{-4}\) & & \(6.951 \times 10^{-4}\) & &
\(4.562 \times 10^{-6}\) & & \(1.479 \times 10^{-3}\) & \\
$200$ & \(1.983 \times 10^{-4}\) & 1.97 & \(1.606 \times 10^{-4}\) & 2.11
& \(1.166 \times 10^{-6}\) & 1.97 & \(3.601 \times 10^{-4}\) & 2.04 \\
$400$ & \(4.995 \times 10^{-5}\) & 1.99 & \(3.966 \times 10^{-5}\) & 2.02
& \(2.943 \times 10^{-7}\) & 1.99 & \(8.990 \times 10^{-5}\) & 2.00 \\
$800$ & \(1.255 \times 10^{-5}\) & 1.99 & \(9.929 \times 10^{-6}\) & 2.00
& \(7.393 \times 10^{-8}\) & 1.99 & \(2.255 \times 10^{-5}\) & 2.00 \\
$1600$ & \(3.145 \times 10^{-6}\) & 2.00 & \(2.491 \times 10^{-6}\) & 2.00
& \(1.853 \times 10^{-8}\) & 2.00 & \(5.654 \times 10^{-6}\) & 2.00 \\
$3200$ & \(7.870 \times 10^{-7}\) & 2.00 & \(6.312 \times 10^{-7}\) & 1.98
& \(4.637 \times 10^{-9}\) & 2.00 & \(1.423 \times 10^{-6}\) & 1.99\\ \bottomrule
  \end{tabular}
\end{table}

\begin{table}\small
  \centering
  \caption{Mesh convergence for the total and the restricted residuals given by \eqref{eq:totalresidual} and \eqref{eq:restricted} in case of the power-law Keller--Segel model in Experiment 3.}\label{tab:exp3}
  \begin{tabular}{rlrlrlrlr}
    \toprule
    mesh & $A^1$ & EOC & $A^2$ & EOC & $A^3$ & EOC & $A$ & EOC\\ \midrule
    $100$ & \(9.221 \times 10^{-4}\) & & \(2.121 \times 10^{-3}\) & &
\(5.473 \times 10^{-6}\) & & \(3.048 \times 10^{-3}\) & \\
$200$ & \(2.368 \times 10^{-4}\) & 1.96 & \(5.383 \times 10^{-4}\) & 1.98
& \(1.414 \times 10^{-6}\) & 1.95 & \(7.765 \times 10^{-4}\) & 1.97 \\
$400$ & \(5.995 \times 10^{-5}\) & 1.98 & \(1.375 \times 10^{-4}\) & 1.97
& \(3.589 \times 10^{-7}\) & 1.98 & \(1.978 \times 10^{-4}\) & 1.97 \\
$800$ & \(1.509 \times 10^{-5}\) & 1.99 & \(3.488 \times 10^{-5}\) & 1.98
& \(9.041 \times 10^{-8}\) & 1.99 & \(5.006 \times 10^{-5}\) & 1.98 \\
$1600$ & \(3.787 \times 10^{-6}\) & 1.99 & \(8.792 \times 10^{-6}\) & 1.99
& \(2.269 \times 10^{-8}\) & 1.99 & \(1.260 \times 10^{-5}\) & 1.99 \\
$3200$ & \(9.483 \times 10^{-7}\) & 2.00 & \(2.215 \times 10^{-6}\) & 1.99
& \(5.682 \times 10^{-9}\) & 2.00 & \(3.169 \times 10^{-6}\) & 1.99 \\ \bottomrule
  \end{tabular}
\end{table}

We moreover show the behavior of the total as well as of the restricted residuals
\begin{equation}\label{eq:restricted}
    A^j(t)= z_1(0)  + \int_0^{T} \| \tilde R^j \|_{H^{-1}}^2  dt, \qquad j\in \{1,2,3\}
\end{equation}
as the mesh is refined in Tables~\ref{tab:exp1}--\ref{tab:exp3}. To analyze mesh convergence we compute the according experimental order of convergence (EOC)\footnote{The EOC is computed by the formula $\text{EOC}=\log_2(A_1/A_2)$ with $A_1$ and $A_2$ denoting the residual in two consecutive lines of the table.}. Throughout the experiments and residual components the results indicate a second order convergence (of the squared $L^2(0,T; H^{-1}(\rT^d))$ norms of all components of the residual and, thus, of the squared discretisation error) with respect to the mesh parameter $n$, which also affects the time step. In addition, the tables show that the first and the second part of the residual resulting from the finite volume discretization of the diffusion and the time discretization are considerably larger than the third part that is due to the discretization of advection.

\section*{Funding}
J.G. is grateful for financial support by the German Science Foundation (DFG) via grant TRR~154 (\emph{Mathematical modelling, simulation and optimization using the example of gas networks}), project C05. The work of J.G. is also supported by the Graduate School CE within Computational Engineering at Technische Universität Darmstadt.
N.K. thanks the German Science Foundation (DFG) for the financial support through project 461365406 and 320021702/GRK2326.
\bibliographystyle{abbrvurl} 
\bibliography{references.bib}

\begin{thebibliography}{10}

\bibitem{Arumgam2021}
G.~Arumugam and J.~Tyagi.
\newblock Keller-{Segel} chemotaxis models: a review.
\newblock {\em Acta Appl. Math.}, 171:82, 2021.
\newblock Id/No 6.
\newblock \href {https://doi.org/10.1007/s10440-020-00374-2}
  {\path{doi:10.1007/s10440-020-00374-2}}.

\bibitem{Bartels}
S.~Bartels.
\newblock {\em Numerical methods for nonlinear partial differential equations},
  volume~47 of {\em Springer Series in Computational Mathematics}.
\newblock Springer, Cham, 2015.
\newblock \href {https://doi.org/10.1007/978-3-319-13797-1}
  {\path{doi:10.1007/978-3-319-13797-1}}.

\bibitem{blanchet2008convermasstrans}
A.~Blanchet, V.~Calvez, and J.~A. Carrillo.
\newblock Convergence of the {{Mass-Transport Steepest Descent Scheme}} for the
  {{Subcritical Patlak}}–{{Keller}}–{{Segel Model}}.
\newblock {\em SIAM J. Numer. Anal.}, 46(2):691--721, Jan. 2008.
\newblock URL: \url{http://epubs.siam.org/doi/10.1137/070683337}, \href
  {https://doi.org/10.1137/070683337} {\path{doi:10.1137/070683337}}.

\bibitem{cai2012fitzhnagum}
S.~Cai, K.~Nagatou, and Y.~Watanabe.
\newblock A numerical verification method for a system of {F}itzhugh-{N}agumo
  type.
\newblock {\em Numer. Funct. Anal. Optim.}, 33(10):1195--1220, 2012.
\newblock \href {https://doi.org/10.1080/01630563.2012.677918}
  {\path{doi:10.1080/01630563.2012.677918}}.

\bibitem{calvez2006volumkellersegel}
V.~Calvez and J.~A. Carrillo.
\newblock Volume effects in the {K}eller-{S}egel model: energy estimates
  preventing blow-up.
\newblock {\em J. Math. Pures Appl. (9)}, 86(2):155--175, 2006.
\newblock \href {https://doi.org/10.1016/j.matpur.2006.04.002}
  {\path{doi:10.1016/j.matpur.2006.04.002}}.

\bibitem{calvez2017equil}
V.~Calvez, J.~A. Carrillo, and F.~Hoffmann.
\newblock Equilibria of homogeneous functionals in the fair-competition regime.
\newblock {\em Nonlinear Anal.}, 159:85--128, 2017.
\newblock \href {https://doi.org/10.1016/j.na.2017.03.008}
  {\path{doi:10.1016/j.na.2017.03.008}}.

\bibitem{Carrillo2022}
J.~A. Carrillo, M.~G. Delgadino, R.~L. Frank, and M.~Lewin.
\newblock Fast diffusion leads to partial mass concentration in
  {Keller}-{Segel} type stationary solutions.
\newblock {\em Math. Models Methods Appl. Sci.}, 32(4):831--850, 2022.
\newblock \href {https://doi.org/10.1142/S021820252250018X}
  {\path{doi:10.1142/S021820252250018X}}.

\bibitem{carrillo2019hybridmasstrans}
J.~A. Carrillo, N.~Kolbe, and M.~{Lukáčová-Medvid’ová}.
\newblock A {{Hybrid Mass Transport Finite Element Method}} for
  {{Keller}}–{{Segel Type Systems}}.
\newblock {\em J. Sci. Comput.}, 80(3):1777--1804, Sept. 2019.
\newblock URL: \url{http://link.springer.com/10.1007/s10915-019-00997-0}, \href
  {https://doi.org/10.1007/s10915-019-00997-0}
  {\path{doi:10.1007/s10915-019-00997-0}}.

\bibitem{chaplain2005mathemmodelcancer}
M.~A.~J. Chaplain and G.~Lolas.
\newblock Mathematical modelling of cancer cell invasion of tissue. {{The}}
  role of the urokinase plasminogen activation system.
\newblock {\em Math. Models Methods Appl. Sci.}, 15(11):1685--1734, Nov. 2005.
\newblock URL: \url{https://www.worldscientific.com/doi/abs/10.1142/
  S0218202505000947}, \href {https://doi.org/10.1142/S0218202505000947}
  {\path{doi:10.1142/S0218202505000947}}.

\bibitem{Chen2022}
W.~Chen, Q.~Liu, and J.~Shen.
\newblock Error estimates and blow-up analysis of a finite-element
  approximation for the parabolic-elliptic {Keller}-{Segel} system.
\newblock {\em Int. J. Numer. Anal. Model.}, 19(2-3):275--298, 2022.
\newblock URL: \url{www.global-sci.org/intro/article_detail/ijnam/20481.html}.

\bibitem{Chertock2018}
A.~Chertock, Y.~Epshteyn, H.~Hu, and A.~Kurganov.
\newblock High-order positivity-preserving hybrid
  finite-volume-finite-difference methods for chemotaxis systems.
\newblock {\em Adv. Comput. Math.}, 44(1):327--350, 2018.
\newblock \href {https://doi.org/10.1007/s10444-017-9545-9}
  {\path{doi:10.1007/s10444-017-9545-9}}.

\bibitem{Chertock2019}
A.~Chertock and A.~Kurganov.
\newblock High-resolution positivity and asymptotic preserving numerical
  methods for chemotaxis and related models.
\newblock In {\em Active particles, Volume 2. Advances in theory, models, and
  applications}, pages 109--148. Cham: Birkh{\"a}user, 2019.
\newblock \href {https://doi.org/10.1007/978-3-030-20297-2_4}
  {\path{doi:10.1007/978-3-030-20297-2_4}}.

\bibitem{CKRW2019}
A.~Chertock, A.~Kurganov, M.~Ricchiuto, and T.~Wu.
\newblock Adaptive moving mesh upwind scheme for the two-species chemotaxis
  model.
\newblock {\em Comput. Math. Appl.}, 77(12):3172--3185, 2019.
\newblock URL: \url{hal.inria.fr/hal-02064581/file/Chertock-Kurganov-Ricchiuto-
  Wu.pdf}, \href {https://doi.org/10.1016/j.camwa.2019.01.021}
  {\path{doi:10.1016/j.camwa.2019.01.021}}.

\bibitem{Dragomir2003}
S.~S. Dragomir.
\newblock {\em Some {G}ronwall type inequalities and applications}.
\newblock Nova Science Publishers, Inc., Hauppauge, NY, 2003.

\bibitem{Dudley2011}
N.~Dudley~Ward, S.~Falle, and M.~S. Olson.
\newblock Modeling chemotactic waves in saturated porous media using adaptive
  mesh refinement.
\newblock {\em Transp. Porous Media}, 89(3):487--504, 2011.
\newblock \href {https://doi.org/10.1007/s11242-011-9782-1}
  {\path{doi:10.1007/s11242-011-9782-1}}.

\bibitem{Epshteyn2009b}
Y.~Epshteyn and A.~Izmirlioglu.
\newblock Fully discrete analysis of a discontinuous finite element method for
  the keller-segel chemotaxis model.
\newblock {\em J. Sci. Comput.}, 40(1-3):211--256, 2009.
\newblock \href {https://doi.org/10.1007/s10915-009-9281-5}
  {\path{doi:10.1007/s10915-009-9281-5}}.

\bibitem{Epshteyn2009a}
Y.~Epshteyn and A.~Kurganov.
\newblock New interior penalty discontinuous galerkin methods for the
  keller–segel chemotaxis model.
\newblock {\em SIAM Journal on Numerical Analysis}, 47(1):386--408, 2009.
\newblock \href {https://doi.org/10.1137/07070423X}
  {\path{doi:10.1137/07070423X}}.

\bibitem{Epshteyn2019}
Y.~Epshteyn and Q.~Xia.
\newblock Efficient numerical algorithms based on difference potentials for
  chemotaxis systems in 3d.
\newblock {\em J. Sci. Comput.}, 80(1):26--59, 2019.
\newblock \href {https://doi.org/10.1007/s10915-019-00928-z}
  {\path{doi:10.1007/s10915-019-00928-z}}.

\bibitem{Filbet2006}
F.~Filbet.
\newblock A finite volume scheme for the {Patlak}-{Keller}-{Segel} chemotaxis
  model.
\newblock {\em Numer. Math.}, 104(4):457--488, 2006.
\newblock \href {https://doi.org/10.1007/s00211-006-0024-3}
  {\path{doi:10.1007/s00211-006-0024-3}}.

\bibitem{GiesselmannKwon}
J.~Giesselmann and K.~Kwon.
\newblock A posteriori error control for a discontinuous galerkin approximation
  of a keller-segel model, in preparation.

\bibitem{Li2019}
L.~Guo, X.~H. Li, and Y.~Yang.
\newblock Energy dissipative local discontinuous {Galerkin} methods for
  {Keller}-{Segel} chemotaxis model.
\newblock {\em J. Sci. Comput.}, 78(3):1387--1404, 2019.
\newblock \href {https://doi.org/10.1007/s10915-018-0813-8}
  {\path{doi:10.1007/s10915-018-0813-8}}.

\bibitem{Horstmann}
D.~Horstmann.
\newblock From 1970 until present: the {Keller}-{Segel} model in chemotaxis and
  its consequences. {II}.
\newblock {\em Jahresber. Dtsch. Math.-Ver.}, 106(2):51--69, 2004.

\bibitem{jager1992explossolutsystem}
W.~Jäger and S.~Luckhaus.
\newblock On explosions of solutions to a system of partial differential
  equations modelling chemotaxis.
\newblock {\em Trans. Amer. Math. Soc.}, 329(2):819--824, Feb. 1992.
\newblock URL:
  \url{http://www.ams.org/jourcgi/jour-getitem?pii=S0002-9947-1992- 1046835-6},
  \href {https://doi.org/10.1090/S0002-9947-1992-1046835-6}
  {\path{doi:10.1090/S0002-9947-1992-1046835-6}}.

\bibitem{kolbe2022adaptrectanmesh}
N.~Kolbe and N.~Sfakianakis.
\newblock An adaptive rectangular mesh administration and refinement technique
  with application in cancer invasion models.
\newblock {\em Journal of Computational and Applied Mathematics}, 416:114442,
  Dec. 2022.
\newblock URL: \url{https://linkinghub.elsevier.com/retrieve/pii/
  S0377042722002096}, \href {https://doi.org/10.1016/j.cam.2022.114442}
  {\path{doi:10.1016/j.cam.2022.114442}}.

\bibitem{kolbe2021model}
N.~Kolbe, N.~Sfakianakis, C.~Stinner, C.~Surulescu, and J.~Lenz.
\newblock Modeling multiple taxis: tumor invasion with phenotypic
  heterogeneity, haptotaxis, and unilateral interspecies repellence.
\newblock {\em Discrete Contin. Dyn. Syst. Ser. B}, 26(1):443--481, 2021.
\newblock \href {https://doi.org/10.3934/dcdsb.2020284}
  {\path{doi:10.3934/dcdsb.2020284}}.

\bibitem{Li2017}
X.~H. Li, C.-W. Shu, and Y.~Yang.
\newblock Local discontinuous {Galerkin} method for the {Keller}-{Segel}
  chemotaxis model.
\newblock {\em J. Sci. Comput.}, 73(2-3):943--967, 2017.
\newblock \href {https://doi.org/10.1007/s10915-016-0354-y}
  {\path{doi:10.1007/s10915-016-0354-y}}.

\bibitem{Liu2018}
J.-G. Liu, L.~Wang, and Z.~Zhou.
\newblock Positivity-preserving and asymptotic preserving method for 2d
  {Keller}-{Segal} equations.
\newblock {\em Math. Comput.}, 87(311):1165--1189, 2018.
\newblock \href {https://doi.org/10.1090/mcom/3250}
  {\path{doi:10.1090/mcom/3250}}.

\bibitem{Liu2001}
W.~Liu and N.~Yan.
\newblock Quasi-norm local error estimators for p-laplacian.
\newblock {\em SIAM Journal on Numerical Analysis}, 39(1):100--127, 2001.
\newblock \href {https://doi.org/10.1137/S0036142999351613}
  {\path{doi:10.1137/S0036142999351613}}.

\bibitem{Makridakis2003}
C.~Makridakis and R.~H. Nochetto.
\newblock Elliptic reconstruction and a posteriori error estimates for
  parabolic problems.
\newblock {\em SIAM J. Numer. Anal.}, 41(4):1585--1594, 2003.
\newblock \href {https://doi.org/10.1137/S0036142902406314}
  {\path{doi:10.1137/S0036142902406314}}.

\bibitem{mizuguchi2017estimsobol}
M.~Mizuguchi, K.~Tanaka, K.~Sekine, and S.~Oishi.
\newblock Estimation of {S}obolev embedding constant on a domain dividable into
  bounded convex domains.
\newblock {\em J. Inequal. Appl.}, pages Paper No. 299, 18, 2017.
\newblock \href {https://doi.org/10.1186/s13660-017-1571-0}
  {\path{doi:10.1186/s13660-017-1571-0}}.

\bibitem{Nicaise2005}
S.~Nicaise.
\newblock A posteriori error estimations of some cell-centered finite volume
  methods.
\newblock {\em SIAM Journal on Numerical Analysis}, 43(4):1481--1503, 2005.
\newblock \href {https://doi.org/10.1137/S0036142903437787}
  {\path{doi:10.1137/S0036142903437787}}.

\bibitem{Saito2007}
N.~Saito.
\newblock {Conservative upwind finite-element method for a simplified
  Keller–Segel system modelling chemotaxis}.
\newblock {\em IMA Journal of Numerical Analysis}, 27(2):332--365, 04 2007.
\newblock \href {https://doi.org/10.1093/imanum/drl018}
  {\path{doi:10.1093/imanum/drl018}}.

\bibitem{Gupta2017}
J.~Sen~Gupta and R.~K. Sinha.
\newblock A posteriori error estimates for lumped mass finite element method
  for linear parabolic problems using elliptic reconstruction.
\newblock {\em Numer. Funct. Anal. Optim.}, 38(12):1527--1547, 2017.
\newblock \href {https://doi.org/10.1080/01630563.2017.1338730}
  {\path{doi:10.1080/01630563.2017.1338730}}.

\bibitem{Sulman2019}
M.~Sulman and T.~Nguyen.
\newblock A positivity preserving moving mesh finite element method for the
  {Keller}-{Segel} chemotaxis model.
\newblock {\em J. Sci. Comput.}, 80(1):649--666, 2019.
\newblock \href {https://doi.org/10.1007/s10915-019-00951-0}
  {\path{doi:10.1007/s10915-019-00951-0}}.

\bibitem{Verfurth_book}
R.~Verf{\"u}rth.
\newblock {\em A posteriori error estimation techniques for finite element
  methods}.
\newblock Numer. Math. Sci. Comput. Oxford: Oxford University Press, 2013.

\bibitem{Zhou2017}
G.~Zhou and N.~Saito.
\newblock Finite volume methods for a {Keller}-{Segel} system: discrete energy,
  error estimates and numerical blow-up analysis.
\newblock {\em Numer. Math.}, 135(1):265--311, 2017.
\newblock \href {https://doi.org/10.1007/s00211-016-0793-2}
  {\path{doi:10.1007/s00211-016-0793-2}}.

\end{thebibliography}
%\bibliography{/Users/niklas/org/bib/global.bib,references.bib}

\end{document}